\documentclass{article}
\usepackage{fullpage}
\usepackage{graphics}
 \usepackage{graphicx}	
\usepackage{amssymb}
 \usepackage{amsmath}
 \usepackage{amsthm} 
 \usepackage{lmodern} 
 \usepackage{hyperref}
\usepackage{pdfsync}
\usepackage{yhmath}
\usepackage{amssymb,mathtools}
\usepackage{stmaryrd}
\usepackage{palatino,tipa,graphicx}

\usepackage{soul}
  \usepackage{xcolor,lmodern}

  \usepackage{calc}
\newsavebox\CBox
\newcommand\hcancel[2][0.5pt]{%
  \ifmmode\sbox\CBox{$#2$}\else\sbox\CBox{#2}\fi%
  \makebox[0pt][l]{\usebox\CBox}%
  \rule[0.5\ht\CBox-#1/2]{\wd\CBox}{#1}}

\newcommand{\tcr}[1]{\textcolor{red}{#1}}

 \newcommand{\dt}{\delta} 
 \renewcommand{\th}{\theta}
 \newcommand{\pd}{\partial} 
 \renewcommand{\H}{\mathcal{H}}

 \newcommand{\ggm}{\Gamma}
 \newcommand{\R}{{\mathbb{R}}}

\newcommand{\fal}{\forall}

\renewcommand{\theequation}{\arabic{section}.\arabic{equation}}

\newcommand{\RR}{{\R^2}}
\newcommand{\EE}{\mathcal{G}}
\newcommand{\Om}{\Omega}
\newcommand{\TT}{{\mathbb{T}^2}}

\newcommand{\eps}{\epsilon}
\newcommand{\NN}{{\mathbb{N}}}
\newcommand{\Peri}{\text{Per}}
\newcommand{\nnn}{\nonumber}

\newcommand{\BVE}{[BV(\TT; \{0,\eta^{-2}\})]^2}

 \newtheorem{theorem}{Theorem}[section]
\newtheorem{lemma}[theorem]{Lemma}

\newtheorem{remark}[theorem]{Remark}

\DeclareMathOperator{\dist}{dist}
\DeclareMathOperator{\diam}{diam}
	
\newcommand{\vep}{\varepsilon} 

\usepackage{stackengine}
\stackMath
\usepackage{mathtools}
\usepackage{esint}

\title{Periodic Minimizers of a \\ Ternary Non-Local Isoperimetric Problem} 

\author{Stanley Alama
\thanks{Department of Mathematics and Statistics, McMaster University. E-mail: alama@mcmaster.ca} 
 \qquad Lia Bronsard 
 \thanks{Department of Mathematics and Statistics, McMaster University. E-mail: bronsard@mcmaster.ca} 
 \qquad Xinyang Lu
 \thanks{Department of Mathematical Sciences, Lakehead University AND Department of Mathematics and Statistics, McGill University. Email: xlu8@lakeheadu.ca}
 \qquad Chong Wang
 \thanks{Department of Mathematics and Statistics, McMaster University. Email: wangc196@mcmaster.ca}
}

\begin{document}

\date{}
\maketitle

\begin{abstract}

We study a two-dimensional ternary inhibitory system derived as a sharp-interface limit of the Nakazawa-Ohta density functional theory of triblock copolymers. This free energy functional combines
an interface energy favoring micro-domain growth with a 
Coulomb-type long range interaction energy which prevents micro-domains from unlimited spreading.  Here we consider a limit in which two species are vanishingly small, but interactions are correspondingly large to maintain a nontrivial limit. In this limit two energy levels are distinguished:  the highest order limit encodes information on the geometry of local structures as a two-component isoperimetric problem,
while the second level describes the spatial distribution of components in global minimizers.  We provide a sharp rigorous derivation of the asymptotic limit, both for minimizers and in the context of Gamma-convergence.  Geometrical descriptions of limit configurations are derived; among other results, we will show 
that, quite unexpectedly, coexistence of single and double bubbles can arise.
The main difficulties are hidden in the
optimal solution of two-component isoperimetric problem: compared to binary systems, not only it lacks an explicit formula, but, more crucially, it can be neither concave nor convex on parts of its domain.

\end{abstract} 

 \numberwithin{equation}{section}

\section{Introduction}

An $ABC$ triblock copolymer is a linear-chain molecule consisting of three subchains, joined covalently to each other. A subchain of type $A$ monomer is connected to one of type $B$, which in turn is connected to another subchain of type $C$ monomer. Because of the repulsive forces between different types of monomers, different types of subchain tend to segregate. However, since subchains are chemically bonded in molecules, segregation can lead to a phase separation only at microscopic level, 
where $A, B$ and $C$-rich micro-domains emerge, forming morphological phases, many of which have been observed experimentally: see Figure \ref{laAndJanus}. Bonding of distinct monomer subchains provides an inhibition mechanism in block copolymers.

\begin{figure}[!htb]
\centering
 \includegraphics[width=3.8cm]{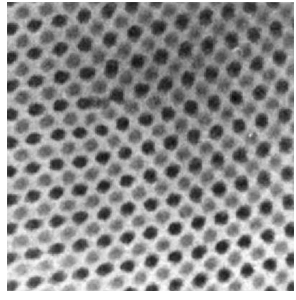} 
 \caption{  Electron microscopy image of the cross-section of a multi-cylinder morphology of ABC triblock copolymers with two cylinder types packed in a tetragonal lattice \cite{Mogi}.
 Reproduced with permission from the American Chemical Society 1994.
 }
 \label{laAndJanus}
\end{figure}

This paper will address the asymptotic behavior of the energy functional derived from  Nakazawa and Ohta's density functional theory for triblock copolymers \cite{microphase, lameRW} in two dimensions.
Let $u = (u_1, u_2)$, and $u_0 = 1 - u_1 - u_2 $. 
The order parameters $u_i, i = 0, 1, 2, $ are defined on $\mathbb{T}^2 = \mathbb{R}^2 / \mathbb{Z}^2=[ - \frac{1}{2}, \frac{1}{2} ]^2$ i.e., the two dimensional flat torus
of unit volume, with periodic boundary conditions. 
Define
\begin{eqnarray} \label{energyu}
\mathcal{E} (u) :=  \frac{1}{2} \sum_{i=0}^2 \int_{\mathbb{T}^2} |\nabla u_i | +  \sum_{i,j = 1}^2  \frac{\gamma_{ij}}{2} \int_{\mathbb{T}^2} \int_{\mathbb{T}^2} G_{\mathbb{T}^2}(x-y)\; u_i (x) \; u_j (y) dx dy
\end{eqnarray}
on $BV(\mathbb{T}^2; \{0,1\})$.
Each $u_i$, which represents the relative monomer density, has two preferred states: $u_i = 0$ and $u_i = 1$. 
Case $u_1 = 1$ corresponds to a pure-$A$ region, $u_2 = 1$ to a pure-$B$ region, and $u_0 = 1$ to a pure-$C$ region.  Thus, each $u_i=\chi_{\Om_i}$ with supports $\Om_i, \ i=0,1,2,$ which partition $\TT$:  $\Om_i$ are assumed to be mutually disjoint
and  $u_1+u_2 + u_0=1$ a.e. on $\TT$.

The energy is minimized under two mass or area constraints
\begin{eqnarray} \label{constrain}
 \frac{1}{| \mathbb{T}^2 |}   \int_{\mathbb{T}^2}  u_i  = M_i,  i = 1, 2.
\end{eqnarray}
Here $M_1$ and $M_2$ are the area fractions of type-$A$ and type-$B$ regions, respectively. Constraints \eqref{constrain} model the fact
 that, during an experiment, the compositions of the molecules do not change.

The first term in  \eqref{energyu} counts the perimeter of the interfaces:  indeed, for $u_i\in BV(\TT; \{0,1\})$,
\[ \int_{\mathbb{T}^2} | \nabla u_i |  : = \sup \left \{  \int_{\mathbb{T}^2}  u_i \  \text{div}  \varphi \ dx:  \varphi = (\varphi_1, \varphi_2)\in C^1 (\mathbb{T}^2 ; \mathbb{R}^2), |\varphi(x)| \leq 1 \right \}, \nonumber
\]
 defines the total variation of  the characteristic function $u_i$.  The factor $\frac12$ acknowledges that each interface between the phases is counted twice in the sum.

The second part of \eqref{energyu} is the long range interaction energy, associated with the connectivity of sub-chains in the triblock copolymer macromolecule.
The long range interaction coefficients $\gamma_{ij}$ form a symmetric matrix $\gamma = [\gamma_{ij}]\in \mathbb{R}^{2\times 2}$. 
Here $G_{\mathbb{T}^2}$ is the zero-mean Green's function for $- \triangle$ on $\mathbb{T}^2$ with periodic boundary conditions, satisfying 
\begin{equation}  \label{GLap}
 -\Delta G_{\mathbb{T}^2}(\cdot - y) = \delta(\cdot-y)- 1
 \text{ in } \mathbb{T}^2; 
  \int_{\mathbb{T}^2} {G_{\mathbb{T}^2} (x - y)} dx=0
 \end{equation}
for each $  y\in \mathbb{T}^2$. In two dimensions, the Green's function $G_{\mathbb{T}^2}$ has the local representation 
\begin{equation} \label{G2def}
G_{\mathbb{T}^2}(x - y)= - \frac{1}{2\pi}\log | x-y | + R_{\mathbb{T}^2} (x - y),
\end{equation}
for $|x-y|<\frac12$.
Here $R_{\mathbb{T}^2}\in C^\infty(\mathbb{T}^2)$ is the regular part of the Green's function.

\begin{figure}[!htb]
\centering
 \includegraphics[width=5.2cm]{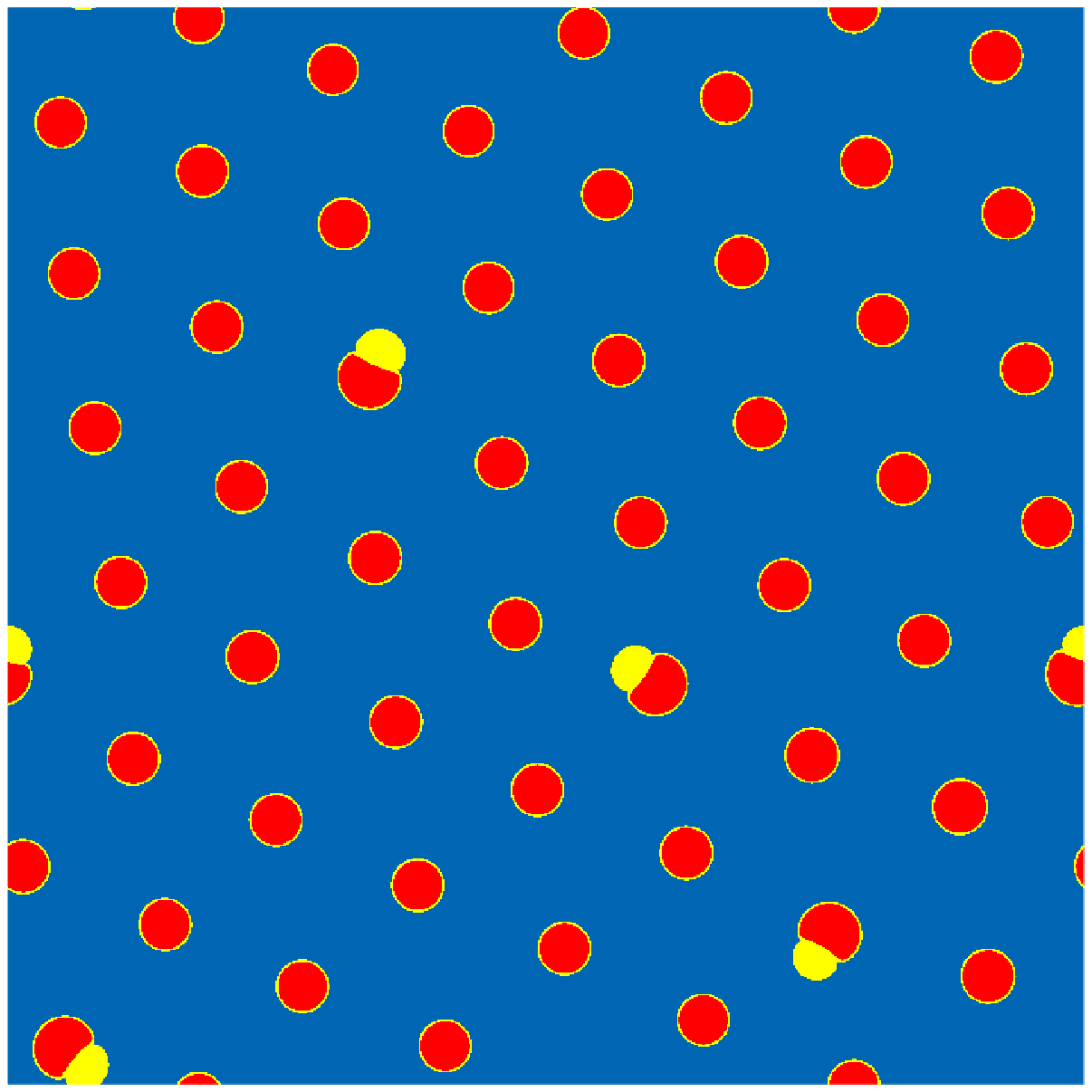} 
\includegraphics[width=5.2cm]{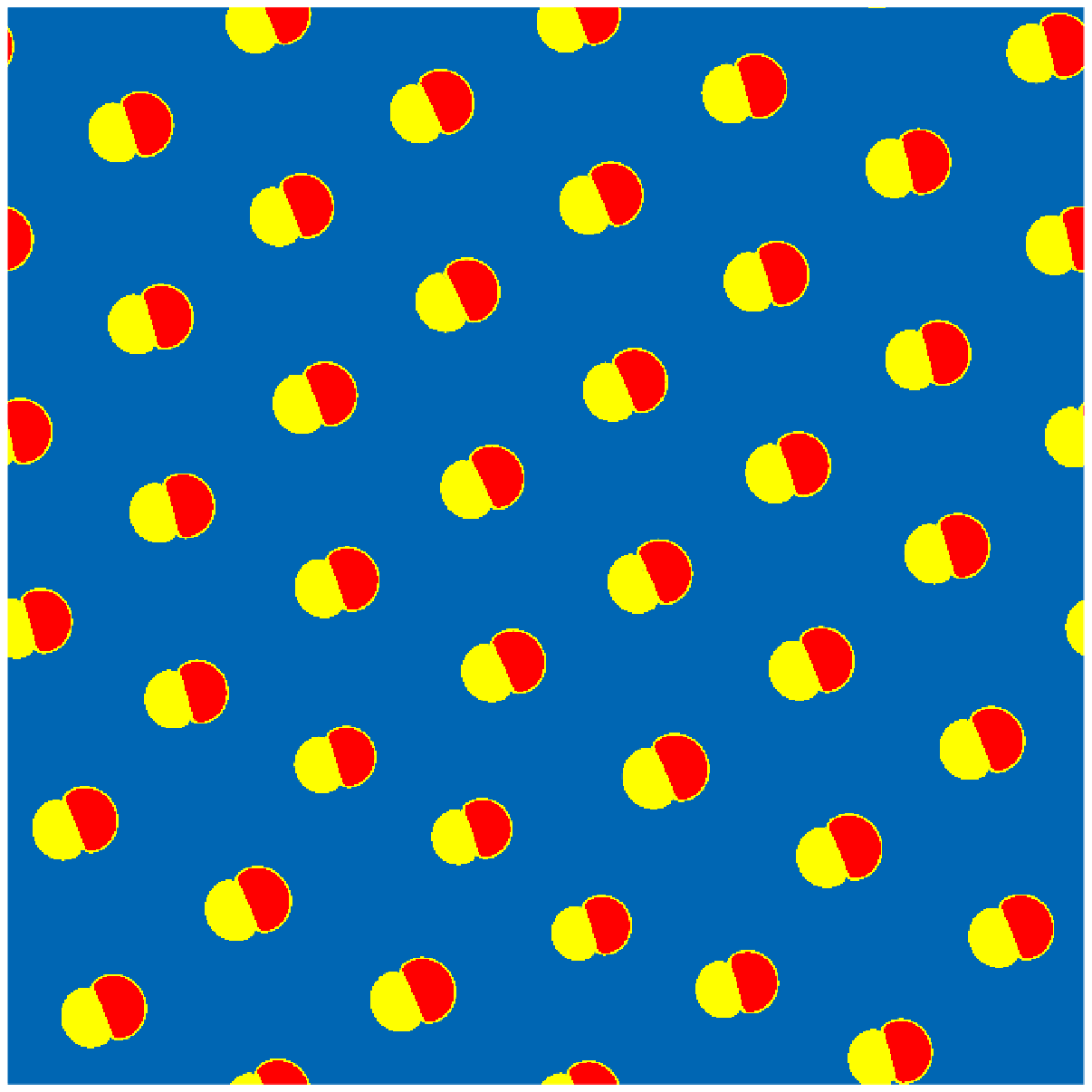} 
\includegraphics[width=5.2cm]{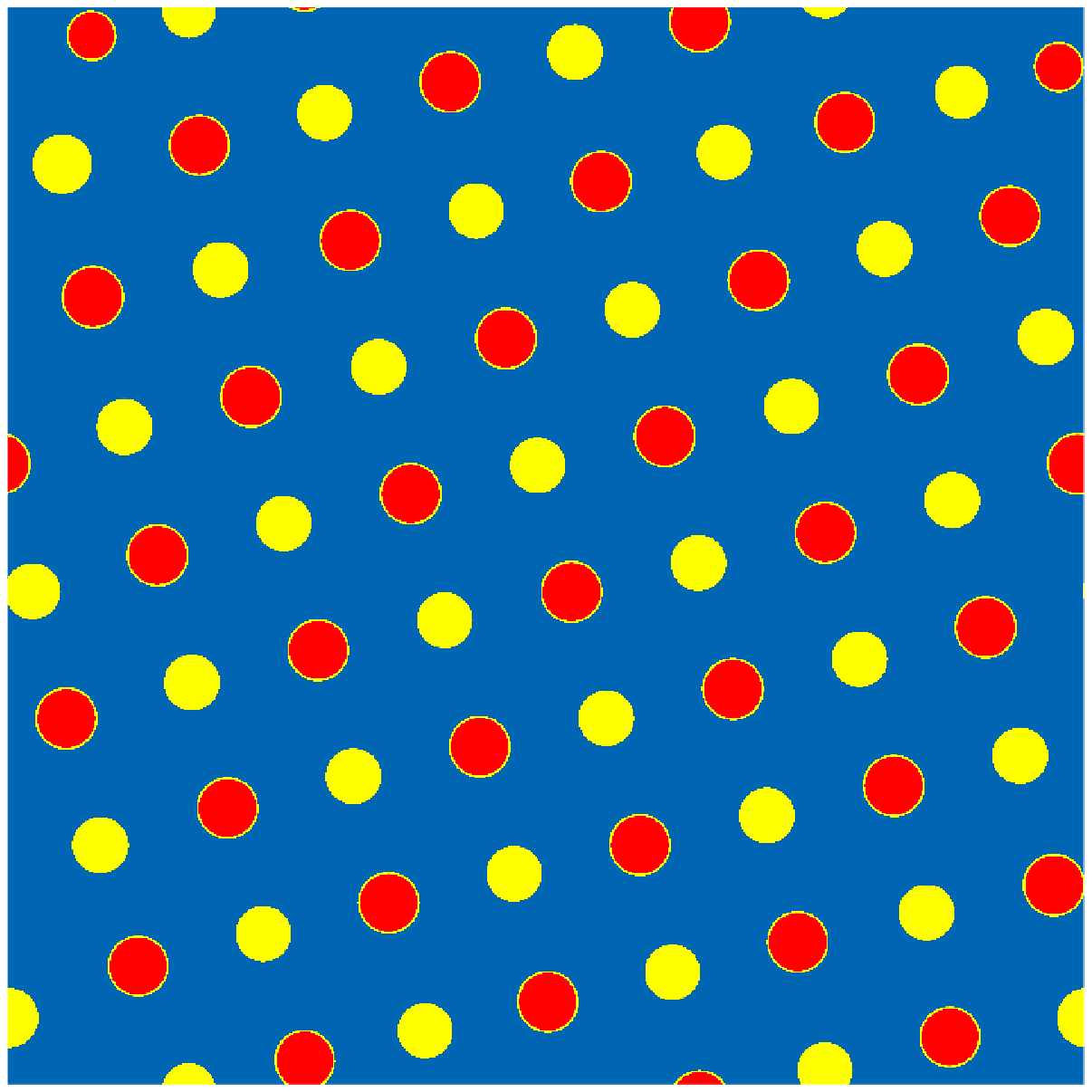}  
 \caption{Numerical simulations: coexistence, all double bubble, and all single bubble patterns of $ABC$ triblock copolymers. Type $A$ micro-domains are in red, and type $B$ are in yellow. The rest of the region is filled by type $C$ monomers, in blue.}
\label{terDisk1}
\end{figure}

As was the case for the Ohta-Kawasaki model of diblock copolymers, nonlocal ternary systems are of high mathematical interest because of the diverse patterns which are expected to be observed by its minimizers. 
 In the same way that the binary nonlocal isoperimetric functional is obtained as a sharp-interface limit of Ohta-Kawasaki, the triblock energy \eqref{energyu} is the sharp-interface limit (in the sense of $\Gamma$-convergence) of a ternary phase-field model introduced by Nakazawa-Ohta (see \cite{rwtri}),
\begin{eqnarray} \label{energyu2}
\mathcal{E}^{\epsilon} (u) :=  \frac{1}{2} \sum_{i=0}^2  \left [ \epsilon \int_{\mathbb{T}^2} |\nabla u_i |^2 dx +\frac{1}{\epsilon} \int_{\mathbb{T}^2} u_i^2 (1- u_i^2) dx \right] +
  \sum_{i,j = 1}^2  \frac{\gamma_{ij}}{2} \int_{\mathbb{T}^2} \int_{\mathbb{T}^2} G_{\mathbb{T}^2}(x-y)\; u_i (x) \; u_j (y) dx dy
\end{eqnarray}
defined on $H^1(\mathbb{T}^2) $.

Just as the diblock copolymer problem may be formulated as a nonlocal isoperimetric problem (NLIP) which partitions space into two components, the triblock model is a  NLIP  based on partitions into three disjoint components. 
 Physicists (see e.g. Bates-Fredrickson \cite{block}) have predicted a wide variety of both two dimensional and three dimensional patterns. 
The shape of minimizers is generally believed to come from perimeter minimization, while the nonlocal interactions promote fragmentation.  For partitions of $\mathbb{R}^n$, $n=2,3$, into three components, the minimizers of perimeter are known to be {\it double bubbles} \cite{FABHZ}, \cite{db1}, \cite{db2}.  In addition, there are {\it core shell} configurations (annuli in two or three dimensions) which are non-minimizing critical points of perimeter.  The presence of these additional structures add to the complexity of the energy landscape of the triblock functional.  Figure~\ref{terDisk1} presents numerical simulations of three two-dimensional morphologies.  These are obtained as the $L^2$ gradient flow dynamics of \eqref{energyu2}, which is solved by a semi-implicit Fourier spectral method \cite{wrz}.

A recent series of papers by Ren \& Wei \cite{doubleAs} and Ren \& Wang \cite{disc, stationary} consider the triblock energy in a parameter regime where two of the components are very dilute with respect to the third.  They use perturbative arguments to generate stationary configurations consisting of assemblies of double bubbles, core shells, or single bubbles of both species, in an array.  These solutions are not constructed by minimization, and it is unknown if they are local or global minimizers.  The purpose of this article is to consider global minimizers of the triblock energy in 2D, in an asymptotic regime where two minority phases have vanishingly small area but strong interaction ensures a bounded number of phase domains in the limit.  In particular, we are interested in describing the possible morphologies of minimizers in this dilute limit.

The appropriate ``droplet'' scaling we use was introduced by Choksi \& Peletier \cite{bi1} in the diblock case.  We introduce a new parameter $\eta$ which is to represent the characteristic length scale of the droplet components.  Thus, areas scale as $\eta^2$, and so we choose mass constraints on $u=(u_1,u_2)$, 
\[   \int_{\mathbb{T}^2}  u_i = \eta^2 M_i \]
for some fixed $M_i$, $i=1, 2$.  We then rescale $u_i$ as
\begin{eqnarray}
v_{i, \eta}^{} = \frac{ u_i }{\eta^2}, \quad i=0,1,2, \quad\text{with}\quad { \int}_{\mathbb{T}^2}  v_{i, \eta}  =  M_i, \quad i=1,2.
\end{eqnarray}
The matrix $\gamma=[\gamma_{ij}]$ is also scaled, in such a way that both terms contribute at the same order in $\eta$.  This may be accomplished by choosing
\begin{eqnarray}
\gamma_{ij} = \frac{1}{|\log \eta| \eta^3} \Gamma_{ij},   \nonumber
\end{eqnarray}
with fixed constants $\Gamma_{ij}\ge 0$.  Throughout the paper we will assume
$$  \Gamma_{ii}>0 \quad i=1,2, \qquad \Gamma_{12}\ge 0, \qquad\text{and}\qquad  \Gamma_{11}\Gamma_{22}-\Gamma_{12}^2>0 . $$
The hypotheses $\Gamma_{11},\Gamma_{22}>0$ and $\Gamma_{12}\ge 0$ are essential to our results.  The positivity of the matrix $\Gamma$ is a consequence of the derivation of the model from density functional theory \cite{rwtri}, but can be omitted in most of our results.  However, the nature of minimizers would be quite different if the matrix $\Gamma$ were not positive definite. 
Denote by $v_{\eta} = (v_{1,\eta}, v_{2, \eta})$.  As the supports of the component funtions $v_{i, \eta}$ should be of finite perimeter and disjoint, we will assume $v_\eta$ lies in the space
\begin{equation}\label{Xspace}
  X_\eta:=\left\{ (v_{1,\eta}, v_{2, \eta}) \ | \  \eta^2 v_{i,\eta}\in BV(\TT; \{0,1\}), \  v_{1,\eta}\, v_{2, \eta} = 0 \ \  a.e.\right\}. 
\end{equation}
With these definitions and for $v_\eta\in X_\eta$ we define our functional,
\begin{eqnarray} \label{Eeta}
  E_{\eta}^{} (v_{\eta}) := \frac{1}{\eta} \mathcal{E} ( u ) 
  =
 \frac{\eta}{2} \sum_{i=0}^2 \int_{\mathbb{T}^2} |\nabla v_{i,\eta} | +  \sum_{i,j = 1}^2  \  \frac{  \Gamma_{ij} }{2 |\log \eta| } \int_{\mathbb{T}^2} \int_{\mathbb{T}^2} G_{\mathbb{T}^2}(x - y)  v_{i, \eta}(x) v_{j, \eta}(y) dx dy, \end{eqnarray}
and $E_\eta(v_\eta)=+\infty$ otherwise.

\medskip

Heuristically, we expect that (for large enough $M_i>0$) this choice of parameters will lead to fragmentation of a minimizing sequence $v_\eta = \sum_{k=1}^K v_\eta^k$ into $K$ isolated components, each concentrating at a distinct point $\xi^k\in\TT$ and supported on a pair of sets $(\Om_{1,\eta}^k,\Om_{2,\eta}^k)$ with characteristic length scale $O(\eta)$.  Apart from the vectorial nature of the order parameters, this was the result described in \cite{bi1,ABCT2} in the binary case.  Blowing up at $\eta$-scale, we would express the minimizing components 
$v_{i, \eta}^k(\eta x + \xi^k ) = \eta^{-2} z^k_i (x)$, with limiting profile $z_i^k:=\chi_{A_i^k}$ for pairs of sets $A^k=(A^k_1,A^k_2)$ in $\RR$.  With this as an ansatz, the minimizer $v_\eta$ may be treated as a superposition of point particles,
$$  v_\eta \rightharpoonup \sum_{k=1}^K (m^k_1,m^k_2)\, \delta_{\xi^k}, $$
for $m_i^k=|A_i^k|$, and
a formal calculation yields an expansion of the energy of the form:
\begin{align*}
E_\eta(v_\eta)&=\sum_{k=1}^K \sum_{i=0}^2 \frac{\eta}{2} \int_{\TT}|\nabla v_{i, \eta}^k |
  +{\Gamma_{ij} \over 2 |\log \eta|} \sum_{k,\ell=1}^K \sum_{i,j=1}^2
     \int_{\TT}\int_{\TT} v_{i, \eta}^k (x)\, G_{\mathbb{T}^2} (x-y)\, v_{j, \eta}^\ell(y)\, dx\, dy \\
     & =  \sum_{k=1}^K \sum_{i=0}^2 \frac{1}{2}  \int_{A_i^k}  |\nabla z_{i}^k| 
  +{\Gamma_{ij} \over 2|\log \eta|} \sum_{k,\ell=1}^K \sum_{i,j=1}^2
     \int_{ A_i^k}\int_{ A_j^\ell} G_{\mathbb{T}^2} (\xi^k+ \eta \tilde x - \xi^\ell - \eta \tilde y)\, d \tilde x\, d \tilde y \\
     &= \sum_{k=1}^K \left( \text{Per}_{\RR} (A^k) + \sum_{i,j=1}^2 {\Gamma_{ij}\over 4\pi} |A^k_i|\, |A^k_j|\right) + O(|\log\eta|^{-1}),
\end{align*}
where we define the perimeter of the {2-cluster} (see \cite[Chapter 29] {maggi}) $A^{\color{red} }=(A_1^{ {\color{red} } },A_2^{{ \color{red} }})$ of sets $A_1,A_2\subset\RR$ with $|A_1\cap A_2|=0$ as
 \begin{equation}\label{cluster_per}
\text{Per}_F(A) = \frac{1}{2} \sum_{i=0}^2 {\cal H}^1 (A_i \cap F), \text{ where } A_0 = (A_1 \cup A_2)^C.  
\end{equation}
Thus, to highest order, energy minimization define the {\it shape} of minimizing components $A^k$ at scale $\eta$, as minimizers of an isoperimetric problem for clusters in $\RR$.  We define 
$$  \EE (A) : = \text{Per}_{\mathbb{R}^2 } (A) +  \sum_{i,j=1}^2 \frac{\Gamma_{ij} {m}_i    {m}_j }{4 \pi},  $$
and for given $  m  =(  m_1 , m_2 )$, $ m_i \ge 0$, 
$$  e_0( m ):=\min   \left\{   \EE( A ) \ | \  A =(A_1,A_2) \text{ 2-cluster, with $| A_i  |= m_i $, $i=1,2$}\right \}.  $$
If both $ m_i >0$, $i=1,2$, then \cite{FABHZ} the minimum is attained at a {\it double bubble}, whose geometry is uniquely determined by $m$; that is,
\begin{eqnarray} \label{e0m}
e_0^{} (m)  = p(m_1, m_2) + \sum_{i, j=1}^2 \frac{\Gamma_{ij} m_i m_j }{4\pi}. 
\end{eqnarray}
The expression $p(m_1,m_2)=\text{Per}_{\RR}(A)$ gives the perimeter of the minimizing cluster $A=(A_1,A_2)$ with $m_i=|A_i|$, and represents the total perimeter of a double bubble when $m_1,m_2>0$. 
In the case of $m_1 = 0$ (or $m_2 =0$), the minimizer is a single bubble,
 $p(m_1,0)=2\sqrt{\pi m_1}$ (and similarly for $p(0,m_2)$), and so single bubbles simplify to
\begin{eqnarray} \label{e0s}
\text{ } e_0(m) : = e_0(m_1, 0) =  2\sqrt{\pi m_1}  +  \frac{\Gamma_{11} (m_1)^2 }{4\pi},  \ \text{ or } \, \  e_0(m) := e_0(0, m_2) =  2\sqrt{\pi m_2}  +  \frac{\Gamma_{22} (m_2)^2 }{4\pi}.
\end{eqnarray} 
%
%
Thus, we expect that minimizers of $E_\eta$ will always form an array consisting of single or double bubbles (or both);  no other shapes are expected for the components of $\Om_\eta$.  The spatial distribution of the single or double bubbles on $\TT$ should be determined by the higher order terms in a more detailed energy expansion.

However, this heuristic description says nothing about how the total masses $M=(M_1,M_2)$ (at scale $\eta^2$) are to be divided.  Indeed, if either $ M_i$ is large then it may well happen that total energy is reduced by further splitting into smaller components, so as to decrease the quadratic term in $e_0$.  Following \cite{bi1} we define
  \begin{eqnarray}  \label{mine0}
 \overline{e_0 }(M ) := \inf \left\{ \sum_{k=1}^{\infty} e_0 (m^k ) :  m^k = (m_1^k, m_2^k ),  \ m_i^k \geq 0,\ \sum_{k=1}^{\infty}  m_i^k = M_i, i = 1, 2 \right\},
 \label{e0bar}
 \end{eqnarray}
which effectively allows for splitting of sets with large area.  We remark that the problem \eqref{e0bar} is highly non-convex, and we do not expect uniqueness of minimizers for $\overline{e_0 }(M)$.

\medskip

Our main results confirm the heuristic behavior above, and provide some description of the geometry of the limiting component clusters for minimizers.  In Theorem~\ref{MinThm} we prove that $\overline{e_0}(M)$ indeed determines the distribution of masses and the resulting shapes of the components:
$$   \lim_{\eta\to 0}\min\left\{ E_\eta(v_{\eta}) \ | \ v_{\eta} \in X_\eta, \  \int_\TT v_{\eta} = M\right\} = \overline{e_0}(M).  $$
For large enough masses $M=(M_1,M_2)$ minimizers do split into a finite number $K$ of disjoint components, each of which minimizes $\EE(A)$ upon blow-up at scale $\eta$.  Furthermore, the spatial arrangement of the limiting bubbles is determined by minimization of the interaction energy
$$  \mathcal{F}_K (y^1,\dots,y^K ; \{m^1,\dots,m^K\})= \sum_{k,\ell=1\atop k\neq\ell}^K \sum_{i,j=1}^2  \frac{ \Gamma_{ij} }{2}\, m_i^k\, m_j^\ell\, G_{\TT}(y^k-y^\ell).  $$
Thus, global minimization should indeed produce a crystalline lattice of double and/or single bubbles, as in the stationary assemblies constructed in \cite{doubleAs,disc,stationary}.
Theorem~\ref{MinThm}  provides a more precise statement which gives a fine detailed structure of minimizers $v_\eta$ of $E_\eta$.  In the same section we also show that $E_\eta$ and $\overline{e_0}$ are connected via $\Gamma$-convergence; see Theorem~\ref{twodfirst}, and the interaction energy $\mathcal{F}_K$ arises as a second-level $\Gamma$ limit.  These results both sharpen those for the binary (diblock) case \cite{bi1} and generalize to the more complex triblock model.

The most important and original results concern the minimizers of $\overline{e_0}(M)$.
First, minimizing configurations can contain only a finite number of nontrivial components:

\begin{theorem}(Finiteness)\label{finiteness}
For any $M= (M_1, M_2), \ M_1, M_2 > 0$, a minimizing configuration for $\overline{e_0}(M)$ has finitely many nontrivial components.  That is, there exist $ K<\infty$ and pairs $m^1,\dots, m^K$, with $m^k=(m_1^k,m_2^k)\neq (0,0)$, for which $\overline{e_0}(M)=\sum_{k=1}^K e_0(m^k)$.
\end{theorem}

This is proven in the binary case by \cite{bi1}, using the concavity of the perimeter for small masses.  However, in the ternary case the proof is much more complex as the expression for the perimeter of double bubbles is not explicitly known, and in fact it is unknown whether it is concave for small $m$.

Given the number of parameters appearing in the limiting description $\overline{e_0}(M)$, it is impossible to make a simple statement concerning its minimizers.  Numerical simulations suggest a wide variety of potential morphologies, and we prove the following indeed may be observed for appropriate parameter values.
\begin{theorem}\label{properties}
\begin{enumerate}  
\item[(a)] (Coexistence) 
Given  $K_1$ and $K_2 >0$, and $\ggm_{12}= 0$, there exist $\overline{M}_1$ and $\overline{M}_2$ such that for all $M_1>\overline{M}_1$ and $M_2>\overline{M}_2$ minimizing configurations of \eqref{mine0} have at least $K_1$ double bubbles and $K_2$ single bubbles.
 \item[(b)] (All single bubbles)  There exist constants $M_i^*$,  depending only on $\Gamma_{ii}$, $i=1,2$, such that
for any given $M_1>4 M_1^*$, $M_2> 4 M_2^*$, there exists a threshold $\ggm_{12}^{\ast}$ such that
  for all $\ggm_{12}> \ggm_{12}^{\ast}$, any minimizing configuration of {\eqref{mine0}}
has no double bubbles. Moreover, all single bubbles have the same size (see Lemma \ref{finiteSin}).
\item[(c)] (One double bubble) There exist constants $m_i^*$,  depending only on $\Gamma_{ii}$, $i=1,2$, such that for any given  $M_i< \min\{m_i^*, \pi \ggm_{ii}^{-2/3}\}$, $i=1,2$, and sufficiently small $\ggm_{12}>0$ such that
\[\frac{\ggm_{12}}{2\pi} M_1M_2 +p(M_1,M_2)<2\sqrt{\pi}(\sqrt{M_1}+\sqrt{M_2}),\]
then there is a unique minimizer of 
\eqref{mine0} made of one double bubble. Here $p$ denotes the perimeter (see equation \eqref{e0m} and below).
\end{enumerate}
The specific values of  $m_i^*$, $M_i^*$ and $\ggm_{12}^*$ are given in the proof of Theorem~\ref{properties} and in the lemmas derived in Section~\ref{section 2}. 
 \end{theorem}

These are proven via delicate comparison arguments based on the geometry of double bubbles, in Section~\ref{section 2}.
For small $\Gamma_{12}$ and $|M_1-M_2|$, intuition and numerics suggest that minimizers should consist of all double bubbles, which after all are preferred by the isoperimetric inequality for 2-clusters.  However, the non-explicit nature of the perimeter function for double bubbles makes such intricate comparison arguments very challenging.

In Section~\ref{section 3} we consider the interaction terms of order $|\log \eta|^{-1}$ and prove a second-level $\Gamma$-convergence result, Theorem~\ref{twodsecond}.  Minimizers of the functional $F_0$ defined there will determine the crystalline lattice of the concentration points defined by the limit of minimizers of $v_\eta$.  

\medskip

Although experimentally an almost unlimited number of architectures can be synthetically accessed in ternary systems like triblock copolymers  \cite{block}, the
mathematical study of \eqref{energyu} is still in its early stages, due to its complexity. One-dimensional stationary points to the Euler-Lagrange equations of \eqref{energyu} were found in \cite{lameRW, blendCR}.
Two and three dimensional stationary configurations were studied recently in \cite{double, doubleAs, stationary, disc, evolutionTer}.

While mathematical interest in triblock copolymers via the energy functional \eqref{energyu} is relatively recent, there has been much progress in mathematical analysis of nonlocal binary systems.  Much early work concentrated on the diffuse interface Ohta-Kawasaki density functional theory for diblock copolymers \cite{equilibrium, nishiura, onDerivation},
\begin{eqnarray} \label{energyB}
\mathcal{E}^{} (u) :=  \int_{\mathbb{T}^n} |\nabla u | +    \gamma \int_{\mathbb{T}^n} \int_{\mathbb{T}^n} G_{\mathbb{T}^n}(x-y)\; u (x) \; u (y) dx dy,
\end{eqnarray}
with a single mass or volume constraint. The dynamics for a gradient flow for \eqref{energyB} with small volume fraction were developed in \cite{hnr, gc}. 
All stationary solutions to the Euler-Lagrange equation of \eqref{energyB} in one dimension were known to be local minimizers \cite{miniRW}, and
many stationary points in two and three dimensions have been found that match the morphological phases in diblock copolymers \cite{oshita, many, spherical, oval, ihsan, Julin3, cristoferi, afjm}.
The sharp interface nonlocal isoperimetric problems have been the object of great interest, both for applications and for their connection to problems of minimal or constant curvature surfaces.
Global minimizers of \eqref{energyB}, and the related Gamow's Liquid Drop model describing atomic nuclei, were studied in \cite{otto, muratov, bi1, st, GMSdensity, knupfer1, knupfer2, Julin, ms, fl} for various parameter ranges.
Variants of the Gamow's liquid drop model with background potential or with an anisotropic surface energy replacing the perimeter, are studied in \cite{ABCT1,luotto, cnt}.
Higher dimensions are considered in \cite{BC, cisp}.
Applications of the second variation of \eqref{energyB} and its connections to minimality and $\Gamma$-convergence are to be found in \cite{cs,afm,Julin2}.
 The bifurcation from spherical, cylindrical and lamellar shapes with Yukawa instead of Coulomb interaction has been done in \cite{fall}.
Blends of diblock copolymers and nanoparticles \cite{nano, ABCT2} and blends of diblock copolymers and homopolymers are also studied by \cite{BK,blendCR}.
Extension of the local perimeter term to nonlocal $s$-perimeters is studied in \cite{figalli}.

\section{Geometric Properties of Global Minimizers} \label{section 2}

In this secction we analyze the geometric properties of minimizers of  $\overline{e_0}(M)$.  We recall \eqref{mine0}
\begin{eqnarray} 
\overline{e_0}(M)=
\inf  \left \{ \sum_{k=1}^{\infty}  e_0 (m^k)  : m^k = (m_1^k, m_2^k), \ m_i^k \geq 0, \sum_{k=1}^{\infty} m_i^k = M_i, i = 1, 2 \right \}, \notag
\end{eqnarray}
where $M=(M_1,M_2)$ and \eqref{e0m}
\begin{eqnarray} 
e_0^{} (m)  = p(m_1, m_2) + \sum_{i, j=1}^2 \frac{\Gamma_{ij} m_i m_j }{4\pi}. \notag
\end{eqnarray}
Unfortunately, for double bubbles $p(m_1,m_2)$ admits no such simple formula.
   The following Lemmas \ref{c' sharp}, \ref{all but one}, and \ref{flex}, will help us overcome this difficulty.  We present the proofs in the case of double bubbles; the degenerate single bubble cases are completely analogous and in most cases much simpler.
   

 \begin{figure}[!htb]
\centering
 \includegraphics[width=10cm]{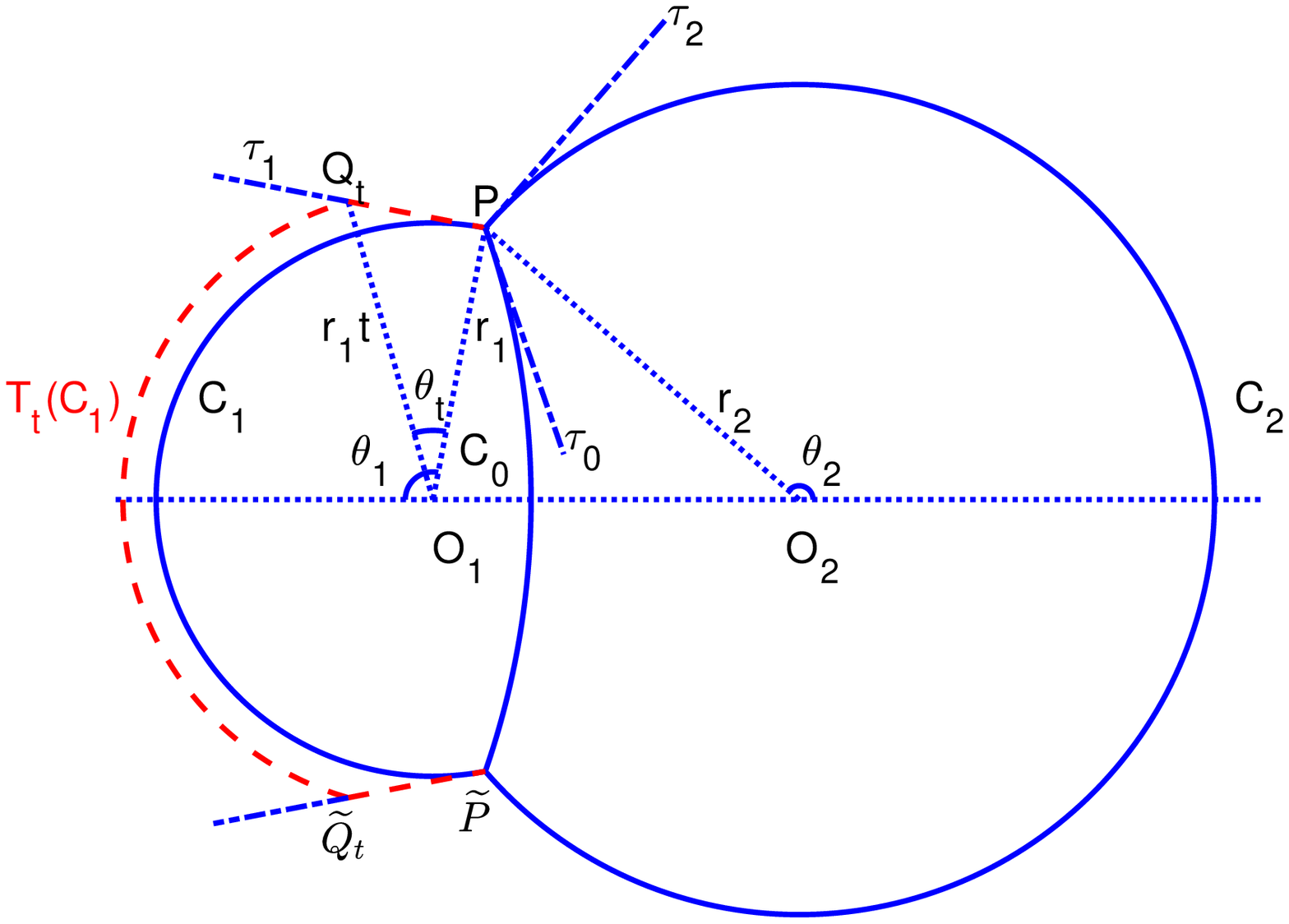} 
  \caption{Construction for the upper bound of $p(m_1+\vep,m_2)$.}
   \label{con1}
\end{figure}

\begin{lemma}\label{c' sharp}
 It holds
 \[\frac{\pd}{\pd m_i} p(m_1,m_2) =\frac{1}{r_i},\qquad i=1,2,\]
 where $r_i = r_i (m_1, m_2)$.
\end{lemma}

\begin{proof}
 Since
 \[\frac{\pd }{\pd m_1}p(m_1,m_2)
 =\lim_{\vep\to 0^+}\frac{p(m_1+\vep,m_2)-p(m_1,m_2)}{\vep}
 =\lim_{\vep\to 0^+}\frac{p(m_1,m_2)-p(m_1-\vep,m_2)}{\vep}, \]
 we need to bound $p(m_1\pm\vep,m_2)$ from above.

Denote by $B$ the double bubble with masses $(m_1,m_2)$.
Denote by $C_i$ the circular arc of the boundary of the lobe with mass $m_i$, radius $r_i$ and center $O_i$, $i=1,2$. 
Also denote by $C_0$ the central arc, by $P$ one of the triple junction points, and by $\tau_i$ the tangent lines
to $C_i$ at $P$, $i=0,1,2$. Being a double bubble, the angle between each two $\tau_i$
and $\tau_j$ with $i\neq j$ is $2\pi/3$. 

Let $T_t(C_1)$ be the scaling of  $C_1$, still centered at $O_1$ and $t>0$ is the ratio.

  \medskip

 {\it Upper bound.}
 We first bound $p(m_1+\vep,m_2)$ from above. 
 To this purpose, it suffices to construct an admissible competitor $B_t$ (which has mass $x+\vep$
 of type I constituent, and mass $m_2$ of type II constituent), which does not need be to be
 a double bubble. We describe only the construction near $P$, since the construction near
 the other triple junction $\widetilde{P}$ will be analogous. In the construction, we do not need to alter the right lobe
 (the one with mass $m_2$); see Figure \ref{con1}.
 
  \begin{itemize}
 \item First, we enlarge $C_1$, replacing it with $T_t(C_1)$ with $t=1+\dt$, for some $\dt=\dt(\vep)$
 that will be determined later.

 \item  We connect the triple junction point $P\in C_0\cup C_2$ to $T_{t}(C_1)$ with the segment 
  $S_t:= \overline{PQ_t}$ where $Q_t:= T_t(C_1)\cap \tau_1$.

 \item Similarly repeat this for the other triple junction point $\widetilde{P}$, which we connect to 
$T_{t}(C_1)$ by a segment $\widetilde{S_t} := \overline{\widetilde{P}\widetilde{Q}_t}$, where $\widetilde{Q}_t$
denotes the reflection
of $Q_t$ with respect to $\overline{O_1O_2}$.
 \end{itemize}
 The competitor will be the region inside 
 \[ B_{t}:=C_0\cup C_2 \cup \wideparen{QQ_t}  \cup S_t \cup \widetilde{S}_t.\]
 Let $\th_t:=\angle P O_1 Q_t$. Note that the triangle
 $\triangle P O_1 Q_t$ satisfies
 \[|O_1-Q_t|=r_1t,\quad  |O_1-P|=r_1, \quad \cos \th_t = \frac{|O_1-P|}{|O_1-Q_t|}=\frac{1}{t},\quad \H^1(S^t)= r_1\tan \th_t.\]
 Choose $t=1+\dt$, with $0<\dt\ll1$, and note that
 \[\cos \th_t = 1-\frac{(\th_t)^2 }{2} +O((\th_t)^4) = \frac{1}{1+\dt} = 1-\dt+o(\dt),\]
 hence $\th_t = \sqrt{2\dt} +o(\sqrt{\dt})$. 
  The piece of arc of $C_1$
 inside $\triangle P O_1 Q_t$ has length $r_1 \th_t$. Thus
\[|\H^1(S_t)-\H^1(C_1\cap \triangle P O_1 Q_t)| = r_1(\tan \th_t-\th_t) 
=r_1\Big( \frac{(\th_t)^3}3+O((\th_t)^5)\Big)=
O(\dt\sqrt{\dt}),\]
and similarly
\[|\H^1(\widetilde{S}_t)-\H^1(C_1\cap \triangle \widetilde{P } O_1 \widetilde{Q}_t)| = O(\dt\sqrt{\dt}).\]
Thus, the difference in perimeter is
\begin{eqnarray*}
&& \H^1(\partial B_{t})-\H^1( \partial B)\\
& = &  \left [\H^1\left(\wideparen{QQ_t}  \right) +
\H^1(S_t)+\H^1(\widetilde{S}_t ) + \H^1(C_0)+\H^1(C_2) \right ] -\left [ \H^1(C_1) + \H^1(C_0)+\H^1(C_2) \right ]\\
&=&  2 r_1 (1+\delta)(\theta_1 - \theta_t)  - 2 r_1 (\theta_1 - \theta_t) + O(\dt\sqrt{\dt}) \\
 &=& 2\th_1r_1\dt+ O(\dt\sqrt{\dt}).
\end{eqnarray*}
Now we need to estimate the difference in area: 
\begin{eqnarray*}
 \H^2(B_t)  - \H^2(B) 
&=&  (\th_1-\th_t ) r_1^2[(1+\dt)^2 -1] + 2 \left [ \H^2(\triangle P O_1 Q^t)-\frac{\th_t r_1^2}2 \right ]\\
&=&  2\th_1 r_1^2 \dt +O(\dt\sqrt{\dt}) + r_1^2 (\tan \th_t -\th_t ) \\
&=& 2\th_1 r_1^2 \dt +O(\dt\sqrt{\dt}).
\end{eqnarray*}
Thus the difference in area between the competitor $B_t$ and the original double bubble
$B$ is
\[2\th_1 r_1^2 \dt +O(\dt\sqrt{\dt}),\]
which has to be equal to $\vep$. Thus $\dt=\vep/2\th_1 r_1^2 +o(\vep)$, and
\[
\lim_{\vep\to 0^+}\frac{p(m_1+\vep,m_2)-p(m_1,m_2)}{\vep}
\le\lim_{\vep\to 0^+}\frac{\H^1(\partial B_t)-p(m_1,m_2)}{\vep} 
=\lim_{\vep\to 0^+}\frac{2\th_1r_1\dt+ O(\dt\sqrt{\dt})}{\vep} 
=\frac1{r_1}.
\]
 \medskip
 
 \begin{figure}[!htb]
\centering
 \includegraphics[width=10cm]{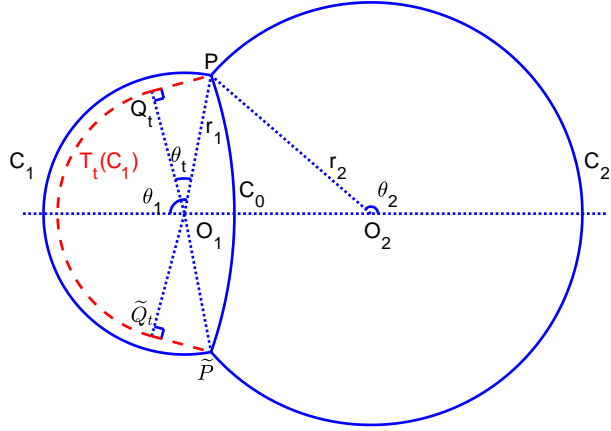} 
  \caption{Construction for the lower bound of $p(m_1-\vep,m_2)$.}
   \label{con2}
\end{figure}
 
 {\it Lower bound.} 
The construction for the lower bound is very similar. Instead of enlarging $C_1$, we now have to shrink it; see Figure \ref{con2}.

 \begin{itemize}
 \item First, we shrink $C_1$, replacing it with $T_t(C_1)$ with $t=1-\dt$, for some $\dt=\dt(\vep)$
 that will be determined later.
 
 \item  
  Let $\theta_t$ be the unique angle such that  the segment 
  $S_t:= \overline{PQ_t}$ is tangent to $T_t(C_1)$ at $Q_t$.

 \item Similarly repeat this for the other triple junction point $\widetilde{P}$, which we connect to 
$T_{t}(C_1)$ by a segment $\widetilde{S_t} := \overline{\widetilde{P}\widetilde{Q}_t}$, where $\widetilde{Q}_t$
denotes the reflection
of $Q_t$ with respect to $\overline{O_1O_2}$.
 \end{itemize}

 Let the competitor be the region inside
 \[ B_{t}:=C_0\cup C_2 \cup \wideparen{QQ_t}  \cup S_t \cup \widetilde{S}_t.\]
Note that our geometric construction gives
\[|O_1-Q_t|=r_1t,\quad \H^1(S_t)=r_1\sin \th_t, \quad \th_t:=\angle PO_1 Q_t = \arccos \frac{|O_1-Q_t|}{|O_1-P|}=t.\]
Choosing $t=1-\dt$ gives again $\th_t:= \sqrt{2\dt}+o(\sqrt\dt)$.

The difference in perimeter is thus
\begin{eqnarray*}
&&\H^1( \partial B) -  \H^1(\partial B_{t})\\
& = & \left [ \H^1(C_1) + \H^1(C_0)+\H^1(C_2) \right ] - \left [\H^1\left(\wideparen{QQ_t}  \right) +
\H^1(S_t)+\H^1(\widetilde{S}_t ) + \H^1(C_0)+\H^1(C_2) \right ] \\
&=&  2 \theta_1 r_1 -2 (\theta_1-\theta_t)r_1(1-\delta) - 2r_1 \sin \theta_t + O(\dt\sqrt{\dt}) \\
 &=& 2\th_1r_1\dt+ O(\dt\sqrt{\dt}).
\end{eqnarray*}
And the difference in area is: 
\begin{eqnarray*}
 \H^2(B)  - \H^2(B_t) 
&=&  (\th_1-\th_t ) r_1^2[1-(1-\dt)^2 ] + 2 \left [ \frac{\th_t r_1^2}2 - \H^2(\triangle P O_1 Q^t) \right ]\\
&=&  2\th_1 r_1^2 \dt +O(\dt\sqrt{\dt}) + r_1^2 (\th_t  - \sin \th_t \cos \th_t) \\
&=& 2\th_1 r_1^2 \dt +O(\dt\sqrt{\dt}).
\end{eqnarray*}
Since we need the area difference to be $\vep$, we get $2\th_1 r_1^2 \dt +O(\dt\sqrt{\dt})=\vep$.
 Thus $\dt=\vep/2\th_1 r_1^2 +o(\vep)$, and
\[
\lim_{\vep\to 0^+}\frac{p(m_1,m_2)-p(m_1-\vep,m_2)}{\vep}
\geq \lim_{\vep\to 0^+}\frac{p(m_1,m_2) - \H^1(\partial B_t)}{\vep} 
=\lim_{\vep\to 0^+}\frac{2\th_1r_1\dt+ O(\dt\sqrt{\dt})}{\vep} 
=\frac1{r_1},
\]
concluding the proof.
\end{proof}


\begin{lemma}\label{all but one}
 Consider an arbitrary minimizing configuration $\mathcal{B}$ of \eqref{mine0} containing at least two double bubbles, denoted by $B_k$, $k=1,2,\cdots$.
 Let $m_1^k$ and $m_2^k$ denote the masses of the two lobes of $B_k$. 
 Then
 the pure second derivatives satisfy
  \[\frac{\partial^2 e_0^{}(m_1^k,m_2^k)}{\partial (m_1^k)^2},\ \frac{\partial^2 e_0^{} (m_1^h,m_2^h)}{\partial (m_2^h)^2}  \geq 0 \] 
 for all except at most one such index $k$ (resp. $h$).

\end{lemma}

 \begin{proof}
 For notational convenience in the proof we denote $x_k:=m_1^k$,  $y_k:=m_2^k$. Consider two (arbitrary) different double bubbles $B_k$, $B_h$. Then
  \begin{align*}
   e_0^{} (x_k+\vep,y_k)-e_0^{} (x_k,y_k) & = \vep \frac{\pd e_0^{} (x_k,y_k)}{\pd x_k} +\frac{\vep^2}{2} \frac{\pd^2 e_0^{} (x_k,y_k)}{\pd x_k^2} +O(\vep^3),\\
   e_0^{} (x_h-\vep,y_h)-e_0^{} (x_h,y_h) & = -\vep \frac{\pd e_0^{} (x_h,y_h)}{\pd x_h} +\frac{\vep^2}{2} \frac{\pd^2 e_0^{} (x_h,y_h)}{\pd x_h^2}+O(\vep^3),
  \end{align*}
hence the minimality of $\mathcal{B}$ gives the necessary condition
\begin{align*}
 0&\le e_0^{} (x_k+\vep,y_k)+e_0^{} (x_h-\vep,y_h)+ \sum_{j\geq 1,\ j\neq k,h}e_0^{} (x_j,y_j) -\sum_{j\geq 1}e_0^{} (x_j,y_j) \\
 &=\vep \bigg(\frac{\pd e_0^{} (x_k,y_k)}{\pd x_k}- \frac{\pd e_0^{} (x_h,y_h)}{\pd x_h} \bigg)+\frac{\vep^2}{2}\bigg( \frac{\pd^2 e_0^{} (x_k,y_k)}{\pd x_k^2}+\frac{\pd^2 e_0^{} (x_h,y_h)}{\pd x_h^2}\bigg) +O(\vep^3),
\end{align*}
hence, by the arbitrariness of $\vep$,
\[ \frac{\pd e_0^{} (x_k,y_k)}{\pd x_k}=\frac{\pd e_0^{} (x_h,y_h)}{\pd x_h}, \quad  \frac{\pd^2 e_0^{} (x_k,y_k)}{\pd x_k^2}+ \frac{\pd^2 e_0^{} (x_h,y_h)}{\pd x_h^2} 
\geq 0, \qquad \fal k\neq h.\]
The proof for the pure second derivative in $y_k$ is completely analogous.
\end{proof}
%


 \begin{lemma}\label{flex}
Given $\ggm_{ii}$, there exist constants $m_i^*$, 
such that
\[ \frac{\pd^2 e_0^{} (m_1,m_2)}{\pd m_i^2} <0 \qquad \text{for all } m_i< m_i^*,\ i=1,2,\]
where $m_i^*$ only depends on $\Gamma_{ii}$. 
 \end{lemma}
In particular, it is quite important for our constructions that $m_i^*$ do not depend neither on
$\ggm_{12}$, nor on the total masses $M_i$, $i=1,2$.

\begin{proof}
 We prove the result for $i=1$. The case $i=2$ is completely analogous.
  By Lemma \ref{c' sharp}, we have
  \[\frac{\pd e_0^{} (m_1,m_2)}{\pd m_1} = \frac{\ggm_{11}m_1+\ggm_{12}m_2 }{2\pi} +\frac{1}{r_1},\qquad
  \frac{\pd^2 e_0^{} (m_1,m_2)}{\pd m_1^2} = \frac{\ggm_{11} }{2\pi} +\frac{\pd}{\pd m_1}\frac{1}{r_1},\]
  where $r_1 = r_1 (m_1, m_2)$.
  So we need to show that there exists a threshold $m_1^*$ such that, for any $m_1<m_1^*$,
  \[\frac{\pd}{\pd m_1}\frac{1}{r_1}<-\frac{\ggm_{11}}{2\pi}.\]
 Thus it suffices to show that
\begin{equation}
\lim_{m_1\rightarrow 0}  \frac{\pd}{\pd m_1}\frac{1}{r_1}  =-\infty.
\label{liminf}
\end{equation}
  \begin{figure}[!htb]
\centering
 \includegraphics[width=8cm]{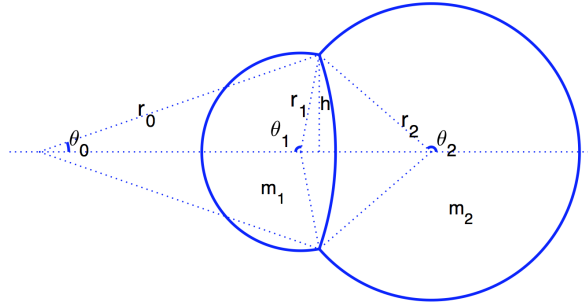}
  \caption{ An asymmetric double bubble with radii $r_i$ and half-angles $\theta_i$, $i=0,1,2$.}
   \label{dbs2}
\end{figure}
 For an asymmetric double bubble bounded by three circular arcs of radii $r_1, r_2$ and $r_0$ with $m_1 < m_2$,
notice that $r_1, r_2, r_0$ and $\theta_1, \theta_2, \theta_0$, the half-angles associated with the three arcs, depend on $m_1$ and $m_2$ implicitly through the equations \cite{isenberg}
\begin{eqnarray}
\label{eq1}
m_1 &=& r_1^2 (\theta_1 - \cos \theta_1 \sin \theta_1) + r_0^2 (\theta_0 - \cos \theta_0 \sin \theta_0),   \\
\label{eq2}
m_2  &=&  r_2^2 (\theta_2 - \cos \theta_2 \sin \theta_2) - r_0^2 (\theta_0 - \cos \theta_0 \sin \theta_0), \\
\label{e3}
h &=&  r_0 \sin \theta_0=  r_1 \sin \theta_1 = r_2 \sin \theta_2,  \\
\label{e5}
 (r_0)^{-1} &=&(r_1)^{-1} - (r_2)^{-1}, \\
\label{e6}
0 &=& \cos \theta_1 + \cos \theta_2 + \cos \theta_0,
\end{eqnarray}
where $r_0$ is the radius of the common boundary of the two lobes of the double bubble;
$\theta_0$ is half of the angle associated with the middle arc; 
and $h$ is half of the distance between two triple junction points; see Figure \ref{dbs2}.
From \eqref{e3} and \eqref{e5}, we have
\begin{eqnarray} \label{e7}
\sin \theta_1 - \sin \theta_2 - \sin \theta_0 = 0.
 \end{eqnarray}
Combine \eqref{e6} with \eqref{e7}, we get
\begin{eqnarray}
\cos(\theta_1+\theta_0) = - \frac{1}{2},  \text{ and } \cos(\theta_2-\theta_0) = - \frac{1}{2}. \nonumber
\end{eqnarray}
That is,
\begin{eqnarray} \label{e8}
\theta_1 = \frac{2\pi}{3} - \theta_0,   \text{ and } \;  \theta_2 = \frac{2\pi}{3} + \theta_0.
\end{eqnarray}

We are interested in the case $m_1 \to 0$. This implies immediately
$h\to 0$, and $r_2\to \sqrt{m_2/\pi}$. Thus $\th_2\to \pi$, hence $\th_0,\th_1\to \pi/3$.
 Let $\th_0=\pi/3-\vep$, $\th_1=\pi/3+\vep$, $\theta_2 = \pi - \vep$. Thus
from \eqref{e3} we get
\begin{equation*}
h=r_2\sin( \pi -\vep) =r_1\sin (\pi/3+\vep) =r_0\sin (\pi/3-\vep).
\end{equation*}
Thus,
\begin{equation} \label{h and r}
r_1=r_2\frac{\sin \vep}{\sin (\pi/3+\vep)},\ r_0=r_2\frac{\sin \vep}{\sin (\pi/3-\vep)}, 
\end{equation}
and \eqref{eq1}, \eqref{eq2} now read
\begin{align}
m_2&= r_2^2 \bigg[\pi- \vep + \frac12 \sin(2\vep)
- \frac{\sin^2 \vep}{\sin^2 (\pi/3-\vep)}\bigg( \frac{\pi}3 -\vep - \frac12\sin \Big(\frac{2\pi}3 -2\vep\Big)\bigg)\bigg]
\label{m2 vep}\\
m_1 &= r_2^2 \left [\frac{\sin^2 \vep}{\sin^2 (\pi/3+\vep)}\bigg( \frac{\pi}3 +\vep - \frac12\sin \Big(\frac{2\pi}3 +2\vep\Big)\bigg)
+ \frac{\sin^2 \vep}{\sin^2 (\pi/3-\vep)}\bigg( \frac{\pi}3 -\vep - \frac12\sin \Big(\frac{2\pi}3 -2\vep\Big)\bigg) \right ]
\notag
\\
&=m_2\frac{\frac{\sin^2 \vep}{\sin^2 (\pi/3+\vep)}( \frac{\pi}3 +\vep - \frac12\sin (\frac{2\pi}3 
+2\vep))
+\frac{\sin^2 \vep}{\sin^2 (\pi/3-\vep)}( \frac{\pi}3 -\vep - \frac12\sin (\frac{2\pi}3 -2\vep))
}{\pi- \vep + \frac12 \sin(2\vep)
- \frac{\sin^2 \vep}{\sin^2 (\pi/3-\vep)}( \frac{\pi}3 -\vep - \frac12\sin (\frac{2\pi}3 -2\vep))}.
\notag
\end{align}
Let
\begin{align*}
D(\vep)&:= \pi- \vep + \frac12 \sin(2\vep)
- \frac{\sin^2 \vep}{\sin^2 (\pi/3-\vep)} \Big( \frac{\pi}3 -\vep - \frac12\sin\Big(\frac{2\pi}3 -2\vep\Big)\Big)
=\pi+O(\vep^2),\\
N(\vep)&:=\frac{\sin^2 \vep}{\sin^2 (\pi/3+\vep)}\Big( \frac{\pi}3 +\vep - \frac12\sin \Big(\frac{2\pi}3 
+2\vep\Big)\Big)
+\frac{\sin^2 \vep}{\sin^2 (\pi/3-\vep)}\Big( \frac{\pi}3 -\vep - \frac12\sin \Big(\frac{2\pi}3 -2\vep\Big)\Big)
=O(\vep^2),
\end{align*}
so we have $m_1=\frac{N(\vep)}{D(\vep)}m_2$, and by direct computation,
\begin{align*}
N'(\vep) &= \frac{\sin 2\vep}{\sin^2 (\pi/3+\vep)}\Big( \frac{\pi}3 +\vep - \frac12\sin \Big(\frac{2\pi}3 
+2\vep\Big)\Big)
+\frac{\sin 2\vep}{\sin^2 (\pi/3-\vep)}\Big( \frac{\pi}3 -\vep - \frac12\sin \Big(\frac{2\pi}3 -2\vep\Big)\Big)\\
&-\frac{2\sin^2 \vep\cos(\pi/3+\vep)}{\sin^3 (\pi/3+\vep)}\Big( \frac{\pi}3 +\vep - \frac12\sin \Big(\frac{2\pi}3 
+2\vep\Big)\Big)\\
&+\frac{2\sin^2 \vep\cos(\pi/3-\vep)}{\sin^3 (\pi/3-\vep)}\Big( \frac{\pi}3 -\vep - \frac12\sin \Big(\frac{2\pi}3 -2\vep\Big)\Big)\\
&+\frac{\sin^2 \vep}{\sin^2 (\pi/3+\vep)}\Big(  1 - \cos \Big(\frac{2\pi}3 +2\vep\Big)\Big)
+\frac{\sin^2 \vep}{\sin^2 (\pi/3-\vep)}\Big(  -1 +\cos \Big(\frac{2\pi}3 -2\vep\Big)\Big)\\
&=\frac{16}{3}\Big(\frac{\pi}3 -\frac{\sqrt{3}}4 \Big) \vep +O(\vep^2),
\end{align*}
and similarly
\begin{align*}
D'(\vep) &=- 1 +  \cos(2\vep)
+\frac{\sin 2\vep}{\sin^2 (\pi/3-\vep)}\Big( \frac{\pi}3 -\vep - \frac12\sin \Big(\frac{2\pi}3 -2\vep\Big)\Big)\\
&+\frac{2\sin^2 \vep\cos(\pi/3-\vep)}{\sin^3 (\pi/3-\vep)}\Big( \frac{\pi}3 -\vep - \frac12\sin \Big(\frac{2\pi}3 -2\vep\Big)\Big)
+\frac{\sin^2 \vep}{\sin^2 (\pi/3-\vep)}\Big(  -1 +\cos \Big(\frac{2\pi}3 -2\vep\Big)\Big)\\
&=\frac{8}{3}\Big(\frac{\pi}3 -\frac{\sqrt{3}}4 \Big)\vep +O(\vep^2).
\end{align*}
Thus
\begin{align*}
\frac1{m_2}\frac{d m_1}{d \vep} = \frac{N'(\vep)D(\vep)-D'(\vep)N(\vep)}{D(\vep)^2}
=\frac{16}{3\pi}\Big(\frac{\pi}3 -\frac{\sqrt{3}}{4} \Big) \vep+O(\vep^2).
\end{align*}

Now we compute the derivative $\frac{\pd r_1}{\pd \vep}$. From \eqref{h and r} and \eqref{m2 vep} 
we get
\[ r_1^2=r_2^2\frac{\sin^2 \vep}{\sin^2 (\pi/3+\vep) } =\frac{\sin^2 \vep}{\sin^2 (\pi/3+\vep) D(\vep)}m_2,\]
hence
\begin{align*}
\frac1{m_2}\frac{\pd r_1^2}{\pd\vep}=\frac{\sin2 \vep}{\sin^2 (\pi/3+\vep) D(\vep)}
- \frac{2\sin^2 \vep \cos(2\pi/3+\vep)}{\sin^3 (\pi/3+\vep) D(\vep)}
-\frac{\sin^2 \vep D'(\vep)}{\sin^2 (\pi/3+ \vep) D^2(\vep)}
= \frac8{3\pi}\vep +O(\vep^2),
\end{align*}
which then gives
\begin{align*}
\frac{\pd r_1}{\pd m_1}&=\frac1{2r_1} \frac{ \frac{\pd r_1^2}{\pd\vep} }{\frac{d m_1 }{d \vep}}
= \frac1{2r_1} \frac{\frac8{3\pi}\vep +O(\vep^2)}{\frac{16}{3\pi}\Big(\frac{\pi}3 -\frac{\sqrt{3}}{4} \Big) \vep+O(\vep^2)}
\ge \frac{C}{r_1}>0,
\end{align*}
for all sufficiently small $\vep<\vep_0$, with $C,\ \vep_0$ being a universal constants independent of $\ggm_{ij}$ and $M_i$, $i,j=1,2$.
Using the fact that in a double bubble we have $\th_1\in (\pi/3,2\pi/3)$, we get
\[ \frac{\pi}{3}r_1^2\le m_1 \le \pi r_1^2,\]
and hence there exists another constant $C'>0$ such that
\[\frac{\pd }{\pd m_1 } \frac{1}{r_1} =-\frac1{r_1^2}\frac{\pd r_1}{\pd m_1}
 \le - \frac{C'}{m_1^{3/2}} \]
as $m_1\to 0$, and \eqref{liminf} is proven. Then note
\[\frac{\pd^2 e_0^{}  (m_1,m_2) }{\pd m_1^2}= \frac{\ggm_{11}}{2\pi}+  \frac{\pd }{\pd m_1 } \frac{1}{r_1}
\le \frac{\ggm_{11}}{2\pi}  - \frac{C'}{m_1^{3/2}},\]
and the proof is complete.

\end{proof}


\bigskip

The next result shows that in a minimizing configuration of \eqref{mine0}, no single bubble, or lobe of double bubbles, can be too large.


\begin{lemma}\label{max mass}
Let $\mathcal{B}$ be a minimizing configuration of \eqref{mine0}. Then
 there exist no single bubble, nor lobe of double bubble, of $i$-th constituent, having mass greater than
 \[M_i^*:= \frac{8 \pi }{\ggm_{ii}^{2/3}},\qquad i=1,2.\]
\end{lemma}

\begin{proof}
 Assume there exists a single bubble of type I material with mass $m$. By replacing it with two single bubbles with mass $m/2$ will change the energy by
 \[ \Delta = 2\Big[\frac{\ggm_{11} m^2}{16\pi}+\sqrt{2\pi m} \Big] -\Big[\frac{\ggm_{11} m^2}{4\pi}+2\sqrt{\pi m} \Big]
 =-\frac{\ggm_{11} m^2}{8\pi}+2\sqrt{\pi m}(\sqrt{2}-1). \]
 The minimality of $\mathcal{B}$ requires $\Delta \ge 0$, which is possible only if
 \[ m^{3/2}\le \frac{16 \pi \sqrt{\pi}(\sqrt{2}-1)}{\ggm_{11}} . \]
 Now assume there exists a lobe of a double bubble of type I material with mass $m$. 
 Denote by $m_2$ the mass of the other lobe (of type II material).
 By removing such lobe and replacing it with two single bubbles with mass $m/2$ will change the energy by
 \begin{align*}
 \Delta & = 2\Big[\frac{\ggm_{11} m^2}{16\pi}+  \sqrt{2\pi m} \Big] +
 \frac{\ggm_{22}  m_2^2}{4\pi}+2\sqrt{\pi m_2}
 -\Big[\frac{\ggm_{11} m^2 +2 \ggm_{12} m m_2+\ggm_{22}  m_2^2}{4\pi}+p(m,m_2) \Big]\\
 &\le-\frac{\ggm_{11} m^2}{8\pi}+2 \sqrt{2\pi m},
 \end{align*}
where we used the fact that $2\sqrt{\pi m_2}=p(0,m_2)\le p(m,m_2)$, which is a direct consequence of Lemma \ref{c' sharp}.
 The minimality of $\mathcal{B}$ requires $\Delta \ge 0$, which is possible only if
 \[ m^{3/2}\le \frac{16 \pi \sqrt{2\pi}}{\ggm_{11}} . \]
 The proof for type II material is completely analogous.
\end{proof}

\begin{lemma}\label{single_lower}
There exist constants $\overline{m_i}^s>0$, $i=1,2$, depending on $\Gamma_{ii}$ only, such that at most one single bubble of $i$-th constituent in a minimizing configuration has mass $m^k_i <\overline{m_i}^s$. 
\end{lemma}

\begin{proof}
Assume $i=1$; the case $i=2$ is the same.  
Assume that for some $M$ there is a minimizing configurations with (at least) two single bubbles of type $I$, whose masses are $m^k_1=x\le y=m^\ell_1$. If we replace this pair by one single bubble with mass $x+y$, the change
in energy may be estimated by:
\begin{align*}
\Delta &= \frac{\ggm_{11} (x+y)^2}{4\pi} + 2\sqrt{\pi(x+y)}- \bigg[\frac{\ggm_{11} (x^2+y^2)}{4\pi} +2\sqrt{\pi}(\sqrt{x}+\sqrt{y}) \bigg]\\
&= \frac{\ggm_{11} xy}{2\pi} +2\sqrt{\pi}(\sqrt{x+y}-\sqrt{x}-\sqrt{y}).
\end{align*}
By the minimality of the given configuration, we must have $\Delta \ge 0$, that is,
\[\frac{\ggm_{11} xy}{2\pi} \ge 2\sqrt{\pi}(\sqrt{x}+\sqrt{y}-\sqrt{x+y}). \]
Thus,
\[ \frac{\ggm_{11} }{4\pi\sqrt{\pi}} xy \ge  \frac{2\sqrt{xy}}{\sqrt{x}+\sqrt{y}+\sqrt{x+y}} 
\ge \frac{\sqrt{xy}}{\sqrt{x}+\sqrt{y}} \overset{(x\le y)}\ge \frac{\sqrt{x}}{2}.\]
Since Lemma \ref{max mass} gives $x,y\le M_1^*$ (which depends on $\ggm_{11}$ only), it follows
\[\frac{\ggm_{11} }{2\pi\sqrt{\pi}} \sqrt{x} M_1^*\ge \frac{\ggm_{11} }{2\pi\sqrt{\pi}} \sqrt{x} y \ge 1,\]
thus there exists a constant $\overline{m_1}^{s} := 4 \pi^3 / ( \Gamma_{11} M_1^{\ast} )^2$ such that  $y\ge x\ge \overline{m_1}^s$.   
\end{proof}

We are now ready to prove Theorem~\ref{finiteness} on the finiteness of minimizing configurations.

\begin{proof}[\bf{Proof of Theorem~\ref{finiteness}}]  
Combining Lemmas \ref{all but one} and \ref{flex}, we conclude that for any minimizing configuration for $\overline{e_0}(M)$ there exists at most one double bubble
whose lobe of the $i$-th constituent has mass less than $m_i^*$, $i=1,2$. Thus there exist at most
\[2+\min\bigg\{\frac{M_1}{m_1^*},\frac{M_2}{m_2^*}\bigg\}\]
double bubbles.

A similar argument based on Lemma \ref{single_lower} yields the same conclusion for single bubbles in a minimizing configuration.
\end{proof}

\medskip


\begin{proof}[\bf{Proof of Theorem \ref{properties} (a)}]
When $\Gamma_{12}= 0$, 
\eqref{e0m} becomes 
\begin{eqnarray} \label{d1}
e_0(m)  =  p(m_1, m_2)  +  \frac{\Gamma_{11} (m_1)^2 }{4\pi} +  \frac{\Gamma_{22} (m_2)^2 }{4\pi}. 
\end{eqnarray}
If there were two single bubbles with different constituents types, \eqref{e0s} would imply that
\begin{eqnarray} \label{s2}
e_0(m_1, 0) + e_0(0, m_2)  =  2\sqrt{\pi m_1}  + 2\sqrt{\pi m_2}  +  \frac{\Gamma_{11} (m_1)^2 }{4\pi} + \frac{\Gamma_{22} (m_2)^2 }{4\pi}.
\end{eqnarray}
Comparing \eqref{d1} with \eqref{s2}, since $p(m_1, m_2) \leq  2\sqrt{\pi m_1}  + 2\sqrt{\pi m_2} $, by \cite{FABHZ} the two single bubbles of different types are more costly than a double bubble of the same masses.
Therefore, all single bubbles must be of the same type of constituent.

Case 1: If all single bubbles (if any) are of type II constituent,
choose $M_1= K_1 M_1^*$, where $M_i^*$ are defined in Lemma~\ref{max mass}. By Lemma \ref{all but one} and Lemma \ref{flex}, there is at most one double bubble whose lobe of type I constituent has mass less than $m_1^*$. 
Combined with Lemma \ref{max mass}, for the other double bubbles, their lobes of type I constituent must have mass between $m_1^*$ and $M_1^*$. Therefore, there are at least $K_1$ double bubbles.

Let $K_d$ be the total number of double bubbles. Clearly $K_d < 1 + M_1/m_1^*$.  Again by Lemmas  \ref{all but one}, \ref{flex} and \ref{max mass}
 there is at most one double bubble whose lobe of type II constituent 
has mass less than $m_2^*$ and for the other double bubbles, their lobes of type II constituent must have mass between  $m_2^*$ and $M_2^*$.
Choose 
 \[M_2 \ge  (1+ M_1 /m_1^*)M_2^* + K_2 \; M_2^* =  (1+(K_1 M_1^*)/m_1^*)M_2^* + K_2 \; M_2^*. \] 
 The type II constituent used by all double bubbles is $K_d M_2^{\ast}$. All the remaining type II constituent must go into single bubbles.
 Therefore, there are at least $K_2$ single bubbles.
 
 Case 2: If all single bubbles (if any) are of type I constituent, via similar arguments, choose $M_2= K_2 M_2^*$ and $M_1 \ge (1+(K_2 M_2^*)/m_2^*)M_1^* + K_1 \; M_1^*$.

 Finally, choose 
 \begin{eqnarray*}
 \overline{M}_1 \ge \max \left \{ K_1 M_1^*,  \left(1+ \frac{K_2 M_2^*}{m_2^*} \right)M_1^* + K_1 \; M_1^*  \right \} = \left(1+ \frac{K_2 M_2^*}{m_2^*} \right)M_1^* + K_1 \; M_1^* , \\
 \overline{M}_2 \ge \max  \left \{ K_2 M_2^*, \left (1+ \frac{K_1 M_1^*}{m_1^*}  \right )M_2^* + K_2 \; M_2^* \right \} = \left (1+ \frac{K_1 M_1^*}{m_1^*}  \right )M_2^* + K_2 \; M_2^*.
 \end{eqnarray*}
Then for all masses $M_1\ge\overline{M}_1$ for the first component and $M_2\ge\overline{M}_2$ for the second, minimizing configurations of \eqref{mine0} have at least $K_1$ double bubbles and $K_2$ single bubbles.

\end{proof}


\begin{lemma} \label{atMostTwo}
Given $\ggm_{11}>0$, $\ggm_{22}>0$, $M_1>0$, $M_2>0$, and
  \[\ggm_{12}>  \frac{4\pi\sqrt{\pi}(\sqrt{M_1^*}+\sqrt{M_2^*}) }{m_1^* m_2^*},\] 
 then any minimizing configuration of \eqref{mine0}
 has at most two double bubbles.
\end{lemma}

\begin{proof}
Consider an arbitrary minimizing configuration $\mathcal{B}$ of \eqref{mine0}. 
 By Lemma \ref{flex}, there exists at most one double bubble with lobe of $i$-th constituent having mass less than $m_i^*$, $i=1,2$. 
 Thus any remaining (if there is)
 double bubble $B_k = (m_1^k, m_2^k )$ satisfying $m_1^k\ge m_1^*$ and $m_2^k\ge m_2^*$. Splitting such a double bubble
 into two single bubbles changes the energy by a quantity
 \begin{align*}
  \Delta &= \sum_{i=1}^2 \frac{\ggm_{ii}(m_i^k)^2}{4\pi} +2\sqrt{\pi}(\sqrt{m_1^k}+\sqrt{m_2^k}) -
  \bigg[ \frac{\ggm_{12} m_1^k m_2^k}{2\pi}+ \sum_{i=1}^2 \frac{\ggm_{ii}(m_i^k)^2}{4\pi}+p(m_1^k, m_2^k) \bigg]\\
  &\le 2\sqrt{\pi}(\sqrt{m_1^k}+\sqrt{m_2^k}) -\frac{\ggm_{12} m_1^k m_2^k}{2\pi}.
 \end{align*}
Using the minimality of $\mathcal{B}$ and Lemma \ref{max mass}, we need
\begin{equation}
0\le \Delta \le 2\sqrt{\pi}(\sqrt{m_1^k}+\sqrt{m_2^k}) -\frac{\ggm_{12} m_1^k m_2^k}{2\pi}\le 2\sqrt{\pi}(\sqrt{M_1^*}+\sqrt{M_2^*}) -\frac{\ggm_{12} m_1^* m_2^*}{2\pi}, 
\label{min gamma}
\end{equation}
and recalling that $M_i^*$ and $m_i^*$ depend only on $\ggm_{ii}$, $i=1,2$, \eqref{min gamma} can hold only when
\[\ggm_{12}\le \frac{4\pi\sqrt{\pi}(\sqrt{M_1^*}+\sqrt{M_2^*}) }{m_1^* m_2^*}.\]
Thus for our choice of $\ggm_{12}$, no double bubble with masses $m_1^k\ge m_1^*$ and $m_2^k\ge m_2^*$
can exist, since splitting it into two single bubbles decreases the energy.

\end{proof}

\medskip

\begin{proof}[\bf{Proof of Theorem \ref{properties} (b)}]
Denote the minimizing configuration of \eqref{mine0} by $\mathcal{B}$.

\begin{itemize}
 \item Claim: there exists constant $\overline{m_i}^d >0$, depending on $\Gamma_{ii}$ only, such that any lobe of $i$-th constituent of a double bubble in $\mathcal{B}$ has mass at least $\overline{m_i}^d$. 
 \end{itemize}
  
Assume there exists a double bubble $D = (x, m_2)$.  Condition $M_1>4M_1^*$, combined with Lemma \ref{atMostTwo},
gives that there are at least two single bubbles
of type I constituent.   By Lemma~\ref{single_lower} there is a single bubble $S$ of type I constituent with mass $m\ge \overline{m_1}^s$.  Removing mass $\vep$ from the lobe of type I constituent of $D$, and adding it to $S$
changes the energy by a quantity
\begin{eqnarray*}
\Delta 
&=&  \left [  \frac{\Gamma_{11} (x -\vep)^2}{4 \pi}  + \frac{\Gamma_{12} (x - \vep) m_2}{ 2 \pi} + \frac{\Gamma_{22} m_2^2}{4\pi} + p(x-\vep, m_2) + \frac{\Gamma_{11} (m+\vep)^2}{4 \pi} + 2\sqrt{\pi (m+\vep)}  \right ]  \\
&& - \left [\frac{\Gamma_{11} x^2 }{4 \pi}  + \frac{\Gamma_{12} x m_2}{ 2 \pi} + \frac{\Gamma_{22} m_2^2}{4\pi} + p(x, m_2) + \frac{\Gamma_{11} m^2}{4 \pi} + 2\sqrt{\pi m } \right  ] \\
&=&  - \frac{ \Gamma_{11} x \vep}{ 2 \pi} -  \frac{ \Gamma_{12}   m_2 \vep }{ 2 \pi} + p(x-\vep, m_2) - p(x,m_2) + \frac{\Gamma_{11} m \vep}{2 \pi} +  2\sqrt{\pi (m+\vep)} -  2\sqrt{\pi m } \\
& =&\bigg(\frac{\ggm_{11}m}{2\pi}+\sqrt{\frac{\pi}{m}}-\frac{\ggm_{11} x +\ggm_{12}m_2}{2\pi}-\frac{1}{r_1} \bigg)\vep+O(\vep^2),
\end{eqnarray*}
where $r_1$ denotes the radius of the lobe of mass $x$. By
 the minimality of $\mathcal{B}$, we need
\begin{equation*}
0\le\frac{\ggm_{11}m}{2\pi}+\sqrt{\frac{\pi}{m}}-\frac{\ggm_{11}x+\ggm_{12}m_2}{2\pi}-\frac{1}{r_1}
\le \frac{\ggm_{11}M_1^*}{2\pi}+\sqrt{\frac{\pi}{\overline{m_1}^s}}-\sqrt{\frac{\pi}{3 x} },
\end{equation*}
and noting that $M_1^*$ and $\overline{m_1}^s$ depend only on $\ggm_{11}$, we get a lower bound on $x$. For any lobe of type II constituent, the proof is the same. Thus, the claim is proven.

\medskip

 Assume there exists a double bubble, whose lobes of $i$-th constituent
have masses $x_i$, $i=1,2$. Splitting such double bubble into two single bubbles (with masses
$x_1$ and $x_2$) changes the energy by
\begin{align*}
 \Delta &= \sum_{i=1}^2\frac{\ggm_{ii}x_i^2}{4\pi} +2\sqrt{\pi}(\sqrt{x_1}+\sqrt{x_2})
 -\bigg[\frac{\ggm_{12}x_1x_2}{2\pi}+ \sum_{i=1}^2\frac{\ggm_{ii}x_i^2}{4\pi}+p(x_1,x_2)\bigg]\\
 &\le2\sqrt{\pi}(\sqrt{M_1^*}+\sqrt{M_2^*})-\frac{\ggm_{12} \overline{m_1}^d \overline{m_2}^d}{2\pi},
\end{align*}
and by the minimality of $\mathcal{B}$, we need
\[0\le 2\sqrt{\pi}(\sqrt{M_1^*}+\sqrt{M_2^*})-\frac{\ggm_{12} \overline{m_1}^d \overline{m_2}^d }{2\pi}.\]
Let
\[ \Gamma_{12}^{\ast} := \frac{4\pi\sqrt{\pi}(\sqrt{M_1^*}+\sqrt{M_2^*}) }{\overline{m_1}^d  \overline{m_2}^d}.  \] 
Since $M_i^*$ and $\overline{m_i}^d$ depend only on $\ggm_{ii}$, $i=1,2$, for all sufficiently
large $\ggm_{12}$, that is, $\ggm_{12} > \ggm_{12}^*$,
no double bubble can exist.
\end{proof}

\medskip

\begin{lemma} (Finite and uniform single bubbles) \label{finiteSin}
Given $\ggm_{11}$, $\ggm_{12}$, $\ggm_{22}$, $M_1$, and $M_2$, if any minimizing configuration of \eqref{mine0}
has only single bubbles, then there are finite single bubbles and all the single bubbles are of the same size.

\end{lemma}

\begin{proof} The proof closely follows that of [\cite{bi1}, Lemma 6.2], for the binary case.  We provide a sketch for the reader's convenience.
By Theorem~\ref{finiteness} there are only finitely many single bubbles. Assume there are $K_1$ type-I single bubbles with masses $\{ m_1^1, m_1^2, \cdots, m_1^{K_1} \}$ and there are $K_2$ type-II single bubbles with masses $\{ m_2^1, m_2^2, \cdots, m_2^{K_2} \}$. 
Let $ F_i(x) =  \frac{\Gamma_{ii} x^2 }{4\pi} + 2\sqrt{\pi x}, $ the contribution to $e_0$ from a single bubble of mass $x$ in the $i^{th}$ constituent.
Note that 
\begin{eqnarray}
F_i(x) = \frac{4\pi}{ (\Gamma_{ii})^{1/3}} f \left(\frac{ x}{4 \pi  (\Gamma_{ii})^{-2/3}} \right ), \nonumber 
\end{eqnarray}
where $f(x) = x^2 + \sqrt{x} $.
Calculations show that $F_i(x)$ is concave on $(0, \pi \ggm_{ii}^{-2/3}] $ and convex on $[ \pi \ggm_{ii}^{-2/3}, \infty) $. 
Thus, for each i, the following can be proved:
\begin{itemize}
 \item[(1).] There is at most one $m_i^k \leq \pi \ggm_{ii}^{-2/3} $.
 \item[(2).] The set of elements $\{ m_i^k: m_i^k \ge \pi \ggm_{ii}^{-2/3} \}$ is a singleton since $F\tcr{_i}(x)$ is convex on $[ \pi \ggm_{ii}^{-2/3}, \infty) $.
 \item[(3).] Any masses of the form $\{ x, \underset{K_i-1}{\underbrace{y,\cdots,y}}
  \}$ with $x < \pi \ggm_{ii}^{-2/3} \leq y$ can not be a minimizer of $F_i$.
\end{itemize} 
\end{proof}


\medskip

\begin{proof}[\bf{Proof of Theorem~\ref{properties} (c)}]
Consider a minimizing configuration $\mathcal{B}$ of \eqref{mine0}.

\medskip

{\em At most one double bubble}: assume the opposite, i.e. there were two double
bubbles, then each lobe (of $i$-th constituent) would have mass less than $m_i^*$, prohibited by Lemma \ref{flex}.

\medskip
{\em At most one single bubble of each constituent}: assume there exist two single bubbles of type I constituent,
with masses $m_1^1$ and $m_1^2$, respectively. Then
\begin{eqnarray*}
&& [ e_0(m_1^1-\vep,0) + e_0(m_1^2 +\vep,0) ]   - [  e_0^{} (m_1^1, 0 ) + e_0^{} (m_1^2, 0) ] \\
 & = & \left [  \frac{\partial}{\partial m_1} e_0^{} (m_1^2, 0) -  \frac{\partial}{ \partial m_1}
 e_0^{} (m_1^1, 0) \right ] \vep +\frac{1}{2}  \left [  \frac{\partial^2}{ (\partial m_1)^2} e_0^{} (m_1^1, 0) +  \frac{\partial^2}{ (\partial m_1)^2}  e_0^{} (m_1^2, 0) \right ] \vep^2  + O (\vep^3)
\end{eqnarray*}
The minimality of $\mathcal{B}$ requires $\frac{\partial}{\partial m_1} e_0^{} (m_1^1, 0 )=  \frac{\partial}{\partial m_1} e_0^{} (m_1^2, 0)$. However,
 since $m_1^1,m_1^2 < M_1 <\min\{m_1^*, \pi \ggm_{11}^{-2/3}\}$, based on the proof in Lemma \ref{finiteSin}, we have 
\[   \frac{\partial^2}{ (\partial m_1)^2}  e_0^{}  (m_1^1, 0 )<0, \qquad   \frac{\partial^2}{ (\partial m_1)^2}  e_0^{} (m_1^2, 0 )<0 ,\]
which is prohibited by the minimality of $\mathcal{B}$.

\medskip

{\em No coexistence}: assume the opposite, i.e. 
there exists a single bubble with mass $m$ which, without loss of generality, we assume made of type I constituent, 
and a double bubble with lobes of masses $m_1$ and $m_2$.
Then
\begin{eqnarray*}
&& [ e_0^{} (m-\vep, 0) + e_0^{} (m_1 +\vep, m_2) ]   - [  e_0^{} (m,0) + e_0^{} (m_1, m_2) ] \\
 & = & \left [  \frac{\partial}{\partial m_1}e_0^{} (m_1,m_2) - \frac{\partial}{\partial m_1} e_0^{} (m, 0) \right ] \vep +\frac{1}{2}  \left [   \frac{\partial^2}{ (\partial m_1)^2}  e_0^{} (m, 0) +  \frac{\partial^2}{ (\partial m_1)^2} e_0^{} (m_1, m_2) \right ] \vep^2 + O (\vep^3)
\end{eqnarray*}
The minimality of $\mathcal{B}$ requires $ \frac{\partial}{\partial m_1} e_0^{} (m, 0)=   \frac{\partial}{\partial m_1} e_0^{} (m_1,m_2) $. However,
 since $m,m_1 < M_1 <\min\{m_1^*, \pi \ggm_{11}^{- 2/3}\}$, based on Lemma \ref{flex} and the proof in Lemma \ref{finiteSin}, we have 
\[ \frac{\partial^2}{ (\partial m_1)^2} \ e_0^{}(m, 0) <0, \qquad  \frac{\partial^2}{ (\partial m_1)^2} \ e_0^{} (m_1, m_2)  < 0 ,\]
which is prohibited by the minimality of $\mathcal{B}$.

\medskip

Finally, we need to compare the case of one double bubble (with lobes of masses
$M_1$ and $M_2$) against the case of two single bubbles of different constituents (of masses $M_1$ and $M_2$).
By our choice of $\ggm_{12}$, we have
\[\frac{\ggm_{12}}{2\pi} M_1M_2 +p(M_1,M_2)<2\sqrt{\pi}(\sqrt{M_1}+\sqrt{M_2}),\]
hence the double bubble has lower energy.

\end{proof}


\section{Convergence Theorems}  \label{section 3}

In this section we formulate and prove two theorems on first-order convergence of $E_\eta$.  

First, we consider global minimizers of $E_\eta$ with given mass condition $\int_{\TT} v_\eta = M$.  Let $v_\eta^{\ast}$ be minimizers of $E_{\eta}$, that is, 
\begin{equation}\label{globmin}
 E_\eta(v_\eta^{\ast})= \min\left\{ E_\eta(v_{\eta}) \ | \ v_{\eta}= ( v_{1, \eta},v_{2, \eta} ) \in X_\eta, \  \int_\TT v_{\eta} = M\right\},
\end{equation}
where the space $X_\eta$ is defined in \eqref{Xspace}.

\begin{theorem}\label{MinThm}
Let  $v_\eta^{\ast}=\eta^{-2}\chi_{\Om_\eta}$ be minimizers of problem \eqref{globmin} for all $\eta>0$.  Then, there exists a subsequence $\eta \to 0$ (still denoted by $\eta$) and $K \in \NN$ such that:
\begin{enumerate}
\item  there exist connected clusters $A^1,\dots, A^K$ in $\RR$ and points $x_{\eta}^k \in \TT$, $k=1,\dots,K$, for which
\begin{equation} \label{MT1}
          \eta^{-2}\left| \Om_\eta \ \triangle \ \bigcup_{k=1}^K \left(\eta A^k + x_{\eta}^k \right)
                          \right| \xrightarrow{\eta \rightarrow 0} 0;
\end{equation}  
\item  each $A^k$, $k=1,\dots,K $ is a minimizer of $\EE$:  
\begin{equation}\label{MT2}
   \EE(A^k)= e_0(m^k), \qquad  m^k=(m_1^k, m_2^k)=|A^k|;
\end{equation}
Moreover,
\begin{equation}\label{MT3} 
   \overline{e_0}(M) = 
   \lim_{\eta\to 0} E_\eta(v_\eta)=\sum_{k=1}^K \EE(A^k)  =     
         \sum_{k=1}^K e_0(m^k).
   \end{equation}   
\item  $x_{\eta}^k \xrightarrow{\eta\rightarrow 0}x^k $, $\forall k=1\dots,K$, and
$\{x^1,\dots,x^K\}$ attains the minimum of 
$\mathcal{F}_K(y^1,\dots,y^K;\{m^1,\dots,m^K\})$ over all $\{y^1,\dots,y^K\}$ in $\TT$.
\end{enumerate}
\end{theorem}
We recall that 
$$  \mathcal{F}_K (y^1,\dots,y^K ; \{m^1,\dots,m^K\})= \sum_{k,\ell=1\atop k\neq\ell}^K  \sum_{i,j=1}^2   \frac{ \Gamma_{ij} }{2}\, m_i^k\, m_j^\ell\, G_{\TT}(y^k-y^\ell).  $$
Thus, minimizers of $E_\eta$ concentrate on a finite number of connected clusters, each of which blows up to a minimizer of the limit energy $\EE$, and with each converging to a {\it different} point in $\TT$.  
Note that an equivalent way to write \eqref{MT1} is using $BV(\TT; \{0,\eta^{-2}\} )$ functions:  
if we define $\Theta_\eta:= \bigcup_{k=1}^K \left(\eta A^k + x_{\eta}^k \right)$ and $w_\eta=\eta^{-2} \chi_{\Theta_\eta}$, then 
$ \Vert v_\eta - w_\eta \Vert_{L^1(\TT)}\xrightarrow{\eta \rightarrow 0} 0$. 
Applying the regularity theory of \cite{maggi} as in \cite{ABCT1, ABCT2} one could show that the convergence actually occurs in a much stronger $C^{1,1}$ sense.  (In which case, we could conclude that minimizers $\Om_\eta$ of $E_\eta$ must also have a bounded number of connected components.)

We note that even in the diblock (binary) case, Theorem~\ref{MinThm} provides a more detailed description of energy minimizers than that of \cite{bi1}.

\medskip

We also formulate the limit in terms of $\Gamma$-convergence.  In this vein we follow the model of \cite{bi1}.  Gamma-limits allow us to consider non-minimizing configurations with energy of the same order as minimizers, and obtain a weaker form of the structure given above.  However, as we will see we can no longer prevent ``coalescence'' of concentration points at the same limit point $\xi\in\TT$ unless we have some second-order information, which is available for global minimizers.

We define a class of measures with countable support on $\TT$,
$$  Y:=\left\{ v_0=\sum_{k=1}^\infty (m_1^k,m_2^k)\, \delta_{x^k} \ | \  m_i^k\ge 0, \ x^k\in\TT \ \text{distinct points}\right\},
$$
and a functional on $Y$,
\begin{eqnarray}  \label{E0first}
E_0^{} (v_0 ) : = 
\left\{  
\begin{array}{rcl}
\sum_{k=1}^{\infty} \overline{e_0} (m^k ),  & &  \text{ if }  v_0\in Y, \\  
 \infty, & & \text{ otherwise. }   
 \end{array}  
\right.   
\end{eqnarray}  

Then we have a first Gamma-convergence result:
 \begin{theorem}[First $\Gamma$-limit] \label{twodfirst}
We have 
\begin{eqnarray}
E_{\eta}^{}  \overset{\Gamma}{\longrightarrow} E_0^{}  \  \text{ as } \ \eta \rightarrow 0. \nonumber
\end{eqnarray}
That is,
\begin{enumerate}
\item  Let $ v_{\eta}\in X_\eta $ be a sequence with $\sup_{\eta>0} E_{\eta}^{}  (v_{\eta} )<\infty$.  Then there exists a subsequence $\eta\to 0$ and $v_0\in Y$ such that 
$ v_{\eta} \rightharpoonup v_{0} $ (in the weak topology of the space of measures),  and 
\begin{eqnarray}
\underset{\eta\rightarrow0}{\lim\inf} E_{\eta}^{}  (v_{\eta} ) \geq E_0^{} (v_{0}). \nonumber 
\end{eqnarray}
\item  Let $v_0\in Y$ with $E_0^{} (v_{0}) < \infty$. Then there exists a sequence $ v_{\eta} \rightharpoonup v_{0} $ weakly as measures, such that 
\begin{eqnarray}
\underset{\eta\rightarrow0}{\lim\sup} E_{\eta}^{}  (v_{\eta} ) \leq E_0^{} (v_{0}). \nonumber 
\end{eqnarray}
\end{enumerate}
\end{theorem}

We may also formulate a second $\Gamma$-convergence result at the level of $|\log\eta|^{-1}$ in the energy, which expresses the interaction energy between components at the minimal energy $\overline{e_0}(M)$.  
For $v_\eta\in X_\eta$, let
\begin{eqnarray}
F_{\eta}^{} (v_{\eta} )  : = |\log \eta | \left [ E_{\eta}^{} (v_{\eta} ) -   \overline{e_0 } \left ( \int_{\mathbb{T}^2} v_{\eta}  \right ) \right ].
\end{eqnarray}
Given the behavior described for minimizers in Theorem~\ref{MinThm}, we expect that for configurations with near-minimal energy $E_\eta(v_\eta)\simeq \overline{e_0}(M)$, $F_\eta(v_\eta)$ should describe the interaction energy of the disjoint clusters making up $v_\eta$.

For $K \in \mathbb{N}$, $m_1^k \geq 0$, $m_2^k \geq 0$ and $(m_1^k)^2 + (m_2^k)^2 > 0 $, the sequence $K \otimes  (m_1^k, m_2^k) $ is defined by
\begin{eqnarray*}
 ( K \otimes  (m_1^k, m_2^k) )^k : = \left \{
 \begin{aligned}
 &(m_1^k, m_2^k),&  1 \leq k \leq K, \\
& (0, 0),&  K+1 \leq k < \infty.
 \end{aligned}
 \right.
\end{eqnarray*}
Let ${\cal{M}}_{M}$ be the set of optimal sequences made of all clusters for the problem \eqref{mine0}:
\begin{eqnarray}
{\cal{M}}_{M} := \Bigg \{ K \otimes  (m_1^k, m_2^k) : K \otimes  (m_1^k, m_2^k)  \text{ minimizes  \eqref{mine0} for } M_i ,  i = 1, 2,  \nonumber \\
 \text{ and  } \  \overline{e_0} (m^k) = e_0^{} (m^k), \  m^k = (m_1^k, m_2^k )  \Bigg \}. \nonumber
\end{eqnarray}
Let $Y_M$ denote the space of all measures $v_0=\sum_{k=1}^K m^k \delta_{x^k}$ with $\{x^1,\dots,x^K\}$ distinct points in $\TT$ and $K\otimes m^k\in \mathcal{M}_M$.
For $v_0\in Y_M$, we define the functional
$$  F_0(v_0)= \sum_{i,j =1}^2    \frac{\Gamma_{ij} }{2}   \left \{  \sum_{k=1}^{K}    \left [   f (m^k_i,m^k_j) +  m_i^k m_j^k R_{\mathbb{T}^2}( 0 )  \right ]   
+   \sum_{k,\ell=1\atop k\neq\ell}^K m_i^k m_j^\ell G_{\mathbb{T}^2}(x_i^k - x_j^\ell ) \right \},
$$
and $F_0(v_0)=+\infty$ otherwise.  The term
$$  f(m^k_i,m^k_j) =\frac{1}{2\pi}  \int_{A^k_i}\int_{A^k_j}  \log { 1 \over |x-y|} dx\, dy, $$
where $A^k$ are the minimizers of $e_0(m^k)$, is determined by the first $\Gamma$-limit, and thus is a constant in $F_0$ and does not change with the locations $x^k$ of the bubbles.

\begin{theorem}[Second $\Gamma$-limit]  \label{twodsecond}
We have 
\begin{eqnarray}
F_{\eta}^{}   \overset{\Gamma}{\longrightarrow} F_0^{}     \  \text{ as } \  \eta \rightarrow 0. \nonumber
\end{eqnarray}
That is,
conditions 1 and 2 of Theorem \ref{twodfirst} hold with $E_{\eta} $ and $E_{0}$ replaced by $F_{\eta} $ and $F_{0} $.
\end{theorem}

The proof of the second $\Gamma$-limit follows that of \cite{bi1}, with some modifications as in the proof of statement 3. of Theorem~\ref{MinThm}, and is left to the reader.  We remark that for minimizers, statement 3.~of Theorem~\ref{MinThm} is stronger than that of Theorem~\ref{twodsecond}, as the control of errors from the order-one Gamma limit is effectively incorporated in the hypothesis that $\sup_{\eta>0}F_\eta(v_\eta)<\infty$.

\subsection*{A concentration lemma}

We begin by making precise the heuristic idea that finite energy configurations $v_\eta$ are composed of a disjoint union of well-separated components of size $\eta$.  Ideally, we would hope that each is connected, but it is sufficient to show that they are separated by open sets of diameter $O(\eta)$.

\begin{lemma}\label{components}
Let $v_\eta=\eta^{-2}\chi_{\Om_\eta}\in X_{\eta}$ with $\eta\int_{\TT} |\nabla v_{\eta} | \le C$.  Then, there exists an at most countable collection $v_{\eta}^k=\eta^{-2}\chi_{\Om_{\eta}^k}\in X_{\eta}$, with clusters $\Om_{\eta}^k=(\Om_{1,\eta}^k, \Om_{2,\eta}^k)$, such that 
\begin{enumerate}
\item[(a)]  $\Om_{i, \eta}^k\cap \Om_{j, \eta}^\ell=\emptyset$, for $k\neq\ell$ and $i,j=1,2$.
\item[(b)]  $v_\eta = \sum_{k=1}^\infty v_\eta^k$ in $X_{\eta}$; in particular,
$$  \int_{\TT} |\nabla v_{i, \eta}| = \sum_{k=1}^\infty \int_{\TT} |\nabla v_{i, \eta}^k|.$$
\item[(c)] There exists $C>0$ with $\diam(\Om_{\eta}^k) \le C\eta$ for all $k\in\NN$.
\end{enumerate}
\end{lemma}

\begin{proof}  Our first step is to identify ``components'' of the clusters $\Om_\eta$.  As we are not assuming these are minimizers, these sets do not have any higher regularity, so we will instead define disjoint open sets $\{\widetilde\Sigma_\eta^k\}_{k\in\NN}$ which disconnect $\Om_\eta$ at the $\eta$-scale.
To this, we first 
let $\rho_\eps(x)=\eps^{-2}\rho(x/\eps)$ be a family of $C^\infty$ mollifiers supported in $B_\eps(0)$, $\int_{\TT}\rho_\eps =1$.  Let $\widetilde\Om_\eta=\Om_{1,\eta}\cup\Om_{2, \eta}$ and $\varphi_\eta=\rho_{\eta^2}*\chi_{\widetilde\Om_\eta}$, which is $C^\infty(\TT)$.  Following \cite[Theorem~3.42]{AFP}, we may choose $t=t(\eta)\in(0,1)$ for which the set $\Sigma_\eta:=\{ x: \ \varphi_\eta(x)>t\}$ is an open set with smooth $\partial\Sigma_\eta \  \subset \  \TT$, with 
$$
\text{Per}_{\TT}(\Sigma_\eta)\le \text{Per}_{\TT}(\Om_\eta) + \eta^2 \le C\eta.  
$$
Moreover (by construction) the Hausdorff distance $d(\Sigma_\eta,\widetilde\Om_\eta)\le \eta^2$.  As $\partial\Sigma_\eta$ is smooth, it decomposes into an at most countable collection of smooth, disjoint, open, connected components, $\Sigma_\eta=\bigcup_{k=1}^\infty \Sigma_\eta^k$, $\Sigma_\eta^k\cap\Sigma_\eta^\ell=\emptyset$, $k\neq\ell$.
As the $\Sigma_\eta^k$ are connected, we also conclude that the diameters,
$$  \sum_{k=1}^\infty \text{diam}_\TT(\Sigma_\eta^k) \le C'\eta  $$
for a constant $C'$ independent of $\eta$.

We next grow our sets $\Sigma_\eta^k$ by considering an $\eta^2$-neighborhood,
$$  \widetilde\Sigma_\eta^k:= \bigcup_{y\in\Sigma_\eta^k} B_{\eta^2}(y),  $$
the $\eta^2$-neighborhoods of the sets $\Sigma_\eta^k$.  In this way, the collection $\{\widetilde\Sigma_\eta^k\}$ covers the original set,
\begin{equation}\label{opencover}
 \Om_\eta, \Sigma_\eta \subset \bigcup_{k=1}^\infty \widetilde\Sigma_\eta^k.
\end{equation}
In expanding $\Sigma_\eta^k$ to $\widetilde\Sigma_\eta^k$, these may no longer be disjoint, but this problem may be overcome by fusing together components which intersect. Indeed, if 
$\widetilde\Sigma_\eta^k\cap \widetilde\Sigma_\eta^\ell \neq\emptyset$ for some $\ell\neq k$, then replace that pair in the list by $\widehat\Sigma_\eta^k:=\widetilde\Sigma_\eta^k\cup\widetilde\Sigma_\eta^\ell$
and reorder (if necessary.) The resulting sets $\{\widetilde\Sigma_\eta^k\}_{k\in\NN}$ will in fact be an at most countable collection of disjoint open sets, each of diameter $O(\eta)$, which cover the clusters $\Om_\eta$.

Now we may define our disjoint components of $\Om_\eta$ via
$$    \Om_\eta^k:= \Om_\eta \cap \widetilde\Sigma_\eta^k.  $$
By \eqref{opencover} we may conclude $\Om_\eta=\bigcup_{k=1}^\infty \Om_\eta^k \, $, and since each $\widetilde\Sigma_\eta^k$ is open and the collection is disjoint, we thus obtain (a), (b) and (c).
\end{proof}

Now that we have decomposed $\Om_\eta$ into disjoint clusters, we describe the limiting structure of the set at the $\eta$-scale, in terms of sets minimizing the blow-up energy $\EE$:

\begin{lemma}\label{ConcLem1}
Let $w_\eta=\eta^{-2}\chi_{\Om_\eta}$ 
 with $\int_{\mathbb{T}^2} w_\eta =M>0$, and $\sup_{\eta>0}E_\eta(w_\eta)<\infty$. Then there exists a subsequence of $\eta \to 0$, 
 points $\{x_{\eta}^k \}$ in $\TT$, and clusters $\{A^k\}$ in $\RR$, such that 
 $\Om_\eta=\bigcup_{k\in\NN}\Om_\eta^k$ satisfies:
\begin{equation}\label{CL12}  \left|  A^k \ \triangle  \,\left( \eta^{-1}\left[ \Om_{\eta}^k -x_{\eta}^k \right]\right) \right| 
   \xrightarrow{\eta_{} \rightarrow 0} 0, \qquad\forall k;
   \end{equation}
Moreover,
\begin{gather}
   \label{CL13}
   M_i=\lim_{\eta_{} \to 0} \sum_{k=1}^\infty  \eta^{-2} |\Om_{i, \eta}^k| = \sum_{k=1}^\infty |A_i^k |, \quad i=1,2, \ \text{and}  \\
   \label{CL14}  \liminf_{\eta_{} \to 0} E_{\eta_{} }(w_{\eta_{} }) \ge
       \sum_{k=1}^\infty \EE(A^k) \ge \overline{e_0}(M).
\end{gather}
\end{lemma}

That is, up to sets of negligible area, $\Om_\eta \simeq \bigcup_{k} \left[x_{\eta}^k +\eta A^k\right]$, a disjoint, at most countable, union of {\it fixed} ($\eta$-independent) clusters scaled by $\eta$.  



\begin{proof}  Applying Lemma~\ref{components} to the set $\Om_\eta$ for each fixed $\eta>0$, we obtain an at most countable disjoint collection $\Om_\eta^k=(\Om_{1,\eta}^k,\Om_{2,\eta}^k)$ of clusters in $\TT$, and corresponding $v^k_\eta=\eta^{-2}(\chi_{\Om_{1,\eta}^k}, \chi_{\Om_{2,\eta}^k})\in\BVE$, satisfying conditions (a), (b), (c) of Lemma~\ref{components}.  We denote by 
$$m_\eta^k=(m_{1,\eta}^k, m_{2,\eta}^k)=\eta^{-2}|\Om^k_{\eta}|=
      (\eta^{-2}|\Om^k_{1,\eta}|\ ,\ \eta^{-2}|\Om^k_{2,\eta}|).  $$
  This may be a finite union of size $N_\eta\in\mathbb{N}$, 
in which case we take $\Om_{\eta}^k=(\emptyset,\emptyset)$ for $k>N_\eta$, and the choice of $x_\eta^k$ is irrelevant.  We also recall that it is possible that only one of $m_{i, \eta}^k> 0$.  As
$$  M=(M_1,M_2)=\sum_{k=1}^\infty m_{\eta}^k,  $$
is either a finite sum or a convergent series, without loss of generality we may assume that each sequence $\{\Om_{\eta}^k \}$ is ordered by decreasing cluster mass:  that is, 
$|m_{\eta}^k |=m_{1, \eta}^k+m_{2, \eta}^k \ge |m_{\eta}^{k+1} |$ holds for all $k$.

From the proof of Lemma~\ref{components} we note that each disjoint cluster $\Om^k_{\eta}\subset\Sigma^k_\eta$, with $\{\Sigma^k_\eta\}_{k\in\NN}$ a collection of disjoint open sets.  This disconnection of $\Om_\eta$ also induces a corresponding disconnection on $\TT\setminus\Om_\eta$, and hence the perimeter of the cluster $\Om_\eta$ (see \eqref{cluster_per}) decomposes as
\begin{equation}\label{per_decomp}
\text{Per}_{\TT}(\Om_\eta)=\sum_{k=1}^\infty \text{Per}_{\TT}(\Om^k_{\eta}).  
\end{equation}

For any $k\in\mathbb{N}$, take any $x_{\eta}^k \in \Om_{1,\eta}^k \cup \Om_{2, \eta}^k$.  
By (c) of Lemma~\ref{components} each individual disjoint cluster $\Om_{\eta}^k $ has bounded diameter, and thus there exists $R>0$ independent of $\eta$ with
\begin{equation}\label{esti}
   \Om_{\eta}^k \subset B_{\eta}^k \subset B_{\eta R}(x_{\eta}^k ).
\end{equation}
An immediate consequence of \eqref{esti} is that we can think of each cluster $\Om_{\eta}^k $ as a subset of $\RR$,  and do a blow-up at scale $\eta$. We define
$A_{\eta}^k  = (A_{1, \eta}^k, A_{2, \eta}^k )$, a cluster in $\mathbb{R}^2$, via 
 $$  A_{i, \eta}^k = \eta^{-1} (\Omega_{i, \eta}^k - x_{\eta}^k ) \subset B_{R} (0 )\subset\RR.  $$
For each $k$, $A_{\eta}^k $ is a uniformly bounded family of finite perimeter sets in $\RR$, and
$$  \text{Per}_{\RR} (A_{\eta}^k ) \le \eta^{-1}\int_{\TT} |\nabla \chi_{\Om_{\eta}^k }|
   \le \eta\int_{\TT} |\nabla w_\eta| \le E_\eta(w_\eta),  $$
for each $\eta>0$.   
By Proposition 29.5 in \cite{maggi}, for each $k$ there exists a subsequence (still denoted by $\eta$) $\eta\to 0$ and a cluster $A^k = (A_{1}^{k}, A_{2}^{k})$ in $\mathbb{R}^2$ with $\lim_{\eta \to 0}|A_{\eta}^{k} \ \triangle A^k |=0$, and for which
\begin{eqnarray}
\text{Per}_{\RR}  (A^k ) \leq \underset{\eta \rightarrow 0}{\liminf}  \ \text{Per}_{\RR} ( A_{\eta}^k ),  \ \  |A^k |  =  \underset{\eta \rightarrow 0}{\lim}  | A_{\eta}^k  | = m^k = (m_1^k, m_2^k ), \nonumber
\end{eqnarray}
Thus there exists a common subsequence, which do not relabel, along which each $A_{\eta}^k \to A^k$ in the above sense.

Next we show \eqref{CL13} holds.
As $M_i = \sum_{k=1}^\infty m_{i,\eta}^k $,
we obtain 
$$  M_i=\lim_{\eta \to 0}\sum_{k=1}^\infty m_{i, \eta}^k \ge \sum_{k=1}^\infty m_i^k, \qquad i=1,2.
$$
To obtain the opposite inequality, let $\eps>0$ be given and $C_0=\sup_{\eta>0} E_\eta(w_\eta)$.  By the convergence of the series above we may choose $N\in\mathbb{N}$ for which both
$$ \sum_{k=N}^\infty m_i^k < {\eps\over 2}, \quad \text{and} \quad m_i^N \le  |m^N| < {\eps^2\pi\over 4 C_0^2}, \qquad i=1,2.  $$ 
Since $m_{i, \eta}^k \to m_i^k $ as $\eta \to 0$, we may choose $\eta_0>0$ so that 
$$   m_{i, \eta}^N < 2m_i^N < {\eps^2\pi\over 2 C_0^2}, \qquad \forall \eta <\eta_0.  $$
Using the ordering  $|m_{\eta}^{k+1} |\le |m_{\eta}^k |$ and the isoperimetric inequality,
\begin{align*}
\sum_{k=N}^\infty m_{i, \eta}^k &\le \sum_{k=N}^\infty \sqrt{|m_{\eta}^k |}\sqrt{m_{i,\eta}^k} \le \sqrt{|m_{\eta}^N |}\sum_{k=N}^\infty \sqrt{m_{i,\eta}^k}
\\
&\le \sqrt{ {|m_{\eta}^N |\over 4\pi}} \sum_{k=N}^\infty \text{Per}_\RR (A_{i, \eta}^k )  \\
&\le \sqrt{ {|m_{\eta}^N|\over 4\pi}}\left[ \eta \int_\TT |\nabla w_\eta | \right]< \sqrt{ {|m_{\eta}^N|\over 4\pi}} C_0  < {\eps\over 2},
\end{align*}
for all $\eta <\eta_0$.  Finally, since $m_{\eta}^k \to m^k$ as $\eta \to 0$, by choosing $\eta_0$ smaller if necessary we have
$$  \sum_{k=1}^{N-1} m_{i, \eta}^k  \le \sum_{k=1}^{N-1} m_i^k + {\eps\over 2}, \qquad \forall \eta <\eta_0.  $$
Thus, for all $\eta <\eta_0$, 
$$  M_i = \sum_{k=1}^\infty m_{i, \eta}^k  < 
  \sum_{k=1}^{N-1} m_{i, \eta}^k  + {\eps\over 2} <
  \sum_{k=1}^{N-1} m_i^k + \eps \le  \sum_{k=1}^\infty m_i^k + \eps,
$$
and \eqref{CL13} is verified.

It remains to calculate the energy in this limit.  Let $\tilde w_{i, \eta}^k =\eta^{-2}\chi_{\Om_{i, \eta}^k }$, $i=1,2$, and $\tilde w_{i,\eta}=\sum_{k\in\NN} \tilde w_{i,\eta}^k $.  
  As the clusters $\Om_{\eta}^k $ are smooth with disjoint closures, the perimeter term decomposes exactly at the $\eta$-scale, and by lower semicontinuity
$$  \eta \int_\TT |\nabla \tilde w_\eta | = \eta^{-1}  \sum_{k\in\NN} \Peri_\TT (\Om_{\eta}^k ) 
   \ge   \sum_{k\in\NN} \Peri_\RR (A^k) +o(1).  $$
The nonlocal terms also split into a double sum:  for $i,j=1,2$,
\begin{gather*}
\int_\TT\int_\TT \tilde w_{i,\eta}(x) G_\TT(x-y) \tilde w_{j,\eta} (y)\, dx\, dy 
= \sum_{k,\ell\in\NN} I_{i,j}^{k,\ell}, \quad\text{with} \\
I_{i,j}^{k,\ell}:= \int_\TT\int_\TT  \tilde w_{i,\eta}^k (x) \, G_\TT(x-y)\, \tilde w_{j, \eta}^{\ell} (y) \, dx\, dy = \int_{A_{i, \eta}^k }\int_{A_{j, \eta}^{\ell}} G_\TT( (x_{\eta}^k +\eta \tilde x) - ( x_{\eta}^{\ell}+\eta \tilde y) ) d\tilde x\, d\tilde y.
\end{gather*}
First consider the self-interaction terms, $\ell=k$.  These have two parts,
\begin{align}\nonumber
I_{i,j}^{k,k} &= \int_{A_{i, \eta}^k }\int_{A_{j, \eta}^k } \left[ {1\over 2 \pi} \log {1\over \eta |\tilde x-\tilde y|} + R_\TT(\eta \tilde x - \eta \tilde y)\right] d\tilde x\, d\tilde y  \\
\label{Ikk}
&={|\log\eta |\over 2\pi} |A_{i, \eta}^k | \, |A_{j, \eta}^k | + J_{i,j}(A_{\eta}^k),
\end{align}
where we define (for clusters $A=(A_1,A_2)$ in $\RR$)
\begin{equation}\label{Jk}
J_{i,j}(A):= \int_{A_{i}} \int_{A_{j}} \left[
   \frac{1}{2\pi} \log{1\over |\tilde x-\tilde y|}+R_{\TT}(\eta(\tilde x-\tilde y))\right]d\tilde x\, d\tilde y.
\end{equation}
As  the regular part of the Green's function $R_{\TT}(x-y)\ge C_1$ is bounded below on $\TT$, and (using \eqref{esti}) $A_{i, \eta}^k \subset B_R(0)$ for all $k$, we obtain the estimate
\begin{align*}\nnn
I_{i,j}^{k,k}
&\ge \left[{|\log\eta | \over 2\pi} + C_1\right] |A_{i, \eta}^k | \, |A_{j, \eta}^k | -
        C_2 |A_{i, \eta}^k |,
\end{align*}   
with $C_2=(2\pi)^{-1}\int_{B_R(0)} |\log|\tilde x-\tilde y||d\tilde x $.  Thus, on diagonal,
\begin{align} \label{CL15} \sum_{k\in\NN} \sum_{i,j=1}^2 \frac{ \Gamma_{ij}}{2} I_{i,j}^{k,k} & \ge 
     \sum_{k\in\NN}\sum_{i,j=1}^2 \Gamma_{ij} {|\log\eta| \over 4\pi} |A_{i, \eta}^k | \, |A_{j, \eta}^k | - O(1) \\
     \label{CL16}
    &= {|\log\eta|\over 4\pi}\sum_{k\in\NN} \sum_{i,j=1}^2 \Gamma_{ij} m_i^k\, m_j^k
       - o(|\log\eta|).
\end{align}

We estimate interaction terms between component clusters in terms of their distance, 
\begin{equation}\label{rdef}
r_{\eta}^{k,\ell} :=  \max\{\dist_\TT(x,y) \ | \ x\in\Om_{\eta}^k, \ y\in\Om_{\eta}^{\ell}),
\end{equation}
   for $k\neq\ell$ with $|\Om_{\eta}^{k}|, |\Om_{\eta}^{\ell}|\neq 0$.
The situation is different depending on whether $r_{\eta}^{k,\ell}$ is bounded from below or not.  In case that (taking a further subsequence if necessary) $r_{\eta}^{k,\ell}\ge \delta_0>0$ for all $\eta$, for $k\neq\ell$ and $i,j=1,2$ we have
\begin{align}
\nnn
I_{i,j}^{k,\ell}&= \eta^{-4}\int_{\Om_{i,\eta}^k } \int_{\Om_{j, \eta}^{\ell}} 
    G_\TT(x-y)\, dx\, dy \\
    \nnn
    &= \eta^{-4} \, |\Om_{i,\eta}^{k}|\, |\Om_{j,\eta}^{\ell}| 
    \left[G_\TT(x_{\eta}^{k}- x_{\eta}^{\ell}) - O(\eta)\right] \\
    \label{CL17}
    &= |A_{i, \eta}^{k}|\, |A_{j, \eta}^{\ell}| \left[G_\TT(x_{\eta}^{k}- x_{\eta}^{\ell}) - O(\eta)\right],
\end{align}
which is of lower order than the self-interaction term.  If $r_{\eta}^{k,\ell}\to 0$ as $\eta\to 0$, we have coalescence of two or more clusters; in this case we estimate:
\begin{align}
\nnn
I_{i,j}^{k,\ell}&= \eta^{-4}\int_{\Om_{i, \eta}^{k}} \int_{\Om_{j, \eta}^{\ell}} 
    G_\TT(x-y)\, dx\, dy \\
    \nnn
    &=  \eta^{-4}\int_{\Om_{i, \eta}^{k}} \int_{\Om_{j, \eta}^{\ell}} 
       \left[ -{1\over 2\pi} \log |x-y| + R_\TT(x-y)\right] dx\, dy \\
       \nnn
    &\ge  \eta^{-4} \, |\Om_{i,\eta}^{k}|\, |\Om_{j, \eta}^{\ell}|
        \left[ {|\log r_{\eta}^{k,\ell}|\over 2\pi} - C \right]
\\  \label{CL18}
    &=|A_{i, \eta}^{k}|\, |A_{j, \eta}^{\ell}| \left[ {|\log r_{\eta}^{k,\ell}| \over 2\pi} - C
    \right].
\end{align}
This term may be of the same order as the self-interaction term, if the distance $r_{\eta}^{k,\ell}$ is comparable to $\eta$, in the sense that $|\log r_{\eta}^{k,\ell}|\sim |\log\eta|$.  Later, we will need to distinguish these cases, but for this lemma we need only note that they are nonnegative apart from remainders of order $o(|\log\eta|)$.  Thus, we have the lower bound,
\begin{align} \nnn
E_\eta (w_\eta ) &\ge E_\eta(\tilde w_\eta) - O(\eta)  \\  \nnn
  &=  \sum_{k\in\NN} \Peri_\RR (A_{\eta}^k)
        +{1\over|\log\eta|}\sum_{k,\ell\in\NN} \sum_{i,j=1}^2  \frac{\Gamma_{ij}}{2} I_{i,j}^{k,\ell} -o(1)\\ \nnn
  &\ge \sum_{k\in\NN} \left[ \Peri_\RR (A^k) + 
             {1\over 4\pi}\sum_{i,j=1}^2 \Gamma_{ij} m_i^k\, m_j^k\right]  -o(1) \\
  &= \sum_{k\in\NN} \EE(A^k) - o(1).   \label{CL19}
\end{align}  
As $\sum_{k\in\NN} m_i^k=M_i$, $i=1,2$, we obtain \eqref{CL14}.
\end{proof}


With the decomposition from Lemma~\ref{ConcLem1} we are now ready to prove the first $\Gamma$-convergence statement.  We note that 
 \begin{eqnarray} \label{e0barInequal}
\overline{e_0 } \left (  \sum_{k=1}^{\infty} m_{}^k \right )   \leq  \sum_{k=1}^{\infty} \overline{e_0 }(m^k),
\end{eqnarray}
 where $m^k  = (m_1^k, m_2^k )$ and $ \sum_{k=1}^{\infty} m_{}^k =(\sum_{k=1}^{\infty} m_1^k, \sum_{k=1}^{\infty} m_2^k) $.

 \subsection*{Proof of the first Gamma limit}
We can now prove $\Gamma$ convergence of $E_\eta\to E_0$.

\begin{proof}[\bf{Proof of Theorem \ref{twodfirst}}]
Let $ v_{\eta} = (v_{1,\eta}, v_{2, \eta}) $ be a sequence such that the energies $E_{\eta}^{}  (v_{\eta} )$ and masses $\int_{\mathbb{T}^2} v_{i, \eta}, i = 1, 2 $ are bounded.  Taking a subsequence (not relabeled,) we may assume $\int_\TT v_{i,\eta}\to M_i$, $i=1,2$.  

\medskip

\noindent  {\sc Compactness and Lower Limit:}  Applying Lemma~\ref{ConcLem1} to $v_\eta$, we obtain at most countably many points $x_\eta^k\in\TT$, clusters $\Om_\eta^k$ in $\TT$ and $A^k$ in $\RR$, satisfying \eqref{CL12}-\eqref{CL14}.  In particular, by \eqref{CL12} we may conclude that 
$$   \eta^{-2}\chi_{\Om_{i,\eta}^k}- m_i^k\, \delta_{x^k_\eta} \rightharpoonup 0, $$
in the sense of measures.  Applying Lemma~5.1 of \cite{bi1} (which is based on a general concentration-compactness result of Lions \cite{lions},) we may conclude that (taking a further subsequence,) $v_\eta\rightharpoonup v_0$ with $v_0=(v_{1,0},v_{2,0})=\sum_{k=1}^\infty (m_1^k,m_2^k)\, \delta_{x^k}$, with distinct $x^k\in\TT$ and $m^k_i\ge 0$, as desired.  From \eqref{CL13} we may conclude that $M_i=\sum_{k\in\NN} m^k_i$.
Applying \eqref{CL14} and \eqref{e0barInequal}, we obtain the lower limit.

\medskip

\noindent {\sc Upper limit:}  The upper bound follows from essentially the same argument as in the proof of Theorem~6.1 of \cite{bi1}; the fact that our $v_\eta$ are supported on 2-clusters does not affect the reasoning.  We provide a brief summary here for completeness, as the bound is important for the proof of Theorem~\ref{MinThm}. 

  Let $v_0 = \sum_{k=1}^{\infty} (m_1^k,m_2^k) \delta_{x^k}$ with $ \{ x^k  \}$ distinct and $ m_i^k \geq 0$, for $k\in\NN$, $i=1,2$, with $M_i=\sum_k m^k_i<\infty$, and $E_{0} (v_{0} ) < \infty$.  As the sums are convergent, we may approximate each by truncation to $K<\infty$ terms, $\tilde{v}_{i,0} = \sum_{k=1}^{ K } m_i^k \delta_{x_i^k}$, and in that case
 \begin{eqnarray}
  E_0 \left ( \tilde{v}_{0} \right )  = \sum_{k=1}^{K} \overline{e_0 }(m^k ) \leq \sum_{k=1}^{\infty} \overline{e_0 }(m^k ) = E_0^{} (v_0 ). \nonumber
 \end{eqnarray}
And also note that $\overline{e_0 }(m^k )$ can be approximated to arbitrary precision by
\begin{eqnarray}
\sum_{\ell=1}^{L_k}  e_0 (m^{k \ell}),  \nonumber
\end{eqnarray}
where $m^{k \ell} = (m_1^{k \ell}, m_2^{k \ell})$ and $\sum_{\ell=1}^{\infty} m_i^{k \ell} = m_i^k$.
Therefore, it is sufficient to construct a sequence $  v_{\eta} \rightharpoonup
  \tilde v_{0 } $ such that 
\begin{eqnarray}
 \limsup_{\eta\to 0} E_{\eta}^{}  (v_{\eta} ) \leq \sum_{k=1}^{K} e_0 (m^k )   \text{ for } \tilde v_{i,0} = \sum_{k=1}^{K} m_i^k \delta_{x_i^k}. \label{reduced} 
\end{eqnarray}
To do this, for each pair $m^k=(m^k_1,m^k_2)$ we let $A^k$ be the cluster in $\RR$ which attains the minimum of the blow-up energy, $e_0(m^k)=\EE(A^k)$, and $z^k:=\chi_{A^k}=(\chi_{A^k_1},\chi_{A^k_2})$.  Choosing $K$ points $\xi^k\in\TT$ with dist${}_\TT (\xi^k,\xi^\ell)\ge \delta>0$ for $k\neq\ell$, we claim that the configuration
$$   v_\eta(x) = \eta^{-2}\sum_{k=1}^K z^k (\eta^{-1}(x-\xi^k))  $$
satisfies \eqref{reduced}.  Indeed, the perimeter term splits exactly as in \eqref{per_decomp}, as well as the self interaction terms,
$$  I^{k,k}_{i,j} = {|\log\eta|\over 2\pi} m^k_i\, m^k_j + o(|\log\eta|).  $$
The interaction terms $I^{k,\ell}_{i,j}$ for $k\neq\ell$ are uniformly bounded in $\eta$, as noted in \eqref{CL17}.  This completes the proof of the first $\Gamma$-limit.

\end{proof}


\subsection*{Minimizers of $E_\eta$}

We now continue towards the proof of Theorem~\ref{MinThm} concerning global minimizers of $E_\eta$.  We next consider configurations whose energy $E_\eta(w_\eta)\le \overline{e_0}(M)$, coinciding  with the minimum value suggested by the $\Gamma$-limit.  With this tighter bound we may obtain more information about the component clusters $\Om_{\eta}^{k}$ and their centers $x_{\eta}^{k}$.  However, note that this is {\em not quite} sufficient to conclude that ``coalescence'' of minimizing clusters cannot occur.

\begin{lemma}\label{ConcLem2}  Let $w_\eta=\eta^{-2}\chi_{\Om_\eta}$ with $\int_{{\mathbb{T}}^2} w_\eta =M>0$, and 
$$   \limsup_{\eta\to 0} E_\eta(w_\eta) \le \overline{e_0}(M).  $$
Then there exists a subsequence (still denoted by) $\eta$, $K \in\NN$, clusters $\Om_{\eta}^{k}\subset\TT$ and $A^k\in\RR$ satisfying \eqref{MT1}, \eqref{MT2}, and \eqref{MT3}.  In addition,  
\begin{equation}\label{CL21}
\limsup_{\eta \to 0} {|\log r_{\eta}^{k,\ell}|\over |\log \eta |} =0, \qquad \forall k\neq\ell.
\end{equation}
\end{lemma}

We recall that $r_\eta^{k,\ell}$ is defined in \eqref{rdef}.

\begin{proof}
The existence of an at most countable collection of disjoint clusters $\Om_{\eta}^{k}$ and their blowup sets $A^k$, with $m^k=|A^k|$, follows from Lemma~\ref{ConcLem1}.  From \eqref{CL14} we obtain 
$$\overline{e_0}(M)\ge  \sum_{k=1}^\infty \EE(A^k)\ge \sum_{k=1}^\infty e_0(m^k) 
\ge \overline{e_0}(M),  $$
and so each term is equal.  In particular, 
$$   \sum_{k\in\NN} \left[ \EE(A^k)-e_0(m^k) \right] =0, $$
and since each term in the sum is non-negative, each must vanish.  Therefore, $\EE(A^k)=e_0(m^k)$ and $A^k$ is a minimizer for each $k$.  By Theorem \ref{finiteness} we conclude that there are only a finite number $K \in\NN$ of nontrivial connected clusters $A^1,\dots,A^K$ in the limit.  Therefore, we have
$$\sum_{k=K+1}^\infty |A_\eta^k|=\sum_{k=K+1}^\infty \eta^{-2}|\Om_\eta^k|\to 0,
$$
 and it suffices to consider the finite union $\bigcup_{k=1}^K \Om_\eta^k$.

To prove \eqref{MT1}, we note that for each component $k$,
$$  \eta^{-2}\left|\Om_{\eta}^{k} \ \triangle \ (\eta A^k + x_{\eta}^{k})\right| = 
         \left| \eta^{-1}(\Om_{\eta}^{k}-x_{\eta}^{k}) \ \triangle \ A^k\right| 
            \xrightarrow{\eta_{} \rightarrow 0} 0,  $$
by \eqref{CL12}.  As the number of components is uniformly bounded, \eqref{MT1} follows.

Finally, we must show that the distance between distinct connected clusters is large relative to $\eta$.  Assume that $\exists \ k\neq\ell$ and $c >0$ with 
${|\log r_{\eta}^{k,\ell}| \over  |\log \eta |}\ge c.$  Returning to the lower bound \eqref{CL19}, we retain the term involving $I_{i,j}^{k,\ell}$ in the third line, and obtain:
\begin{align*}
\overline{e_0}(M) = E_\eta (\tilde w_\eta)+o(1) & \ge \sum_{k=1}^K \EE(A^k) 
     + {1\over 4\pi |\log\eta|}\sum_{i,j=1}^2  \Gamma_{ij} I_{i,j}^{k,\ell} + o(1) \\
    &\ge \overline{e_0}(M) + {c \over 4\pi} \sum_{i,j=1}^2  \Gamma_{ij} m_i^k m_j^\ell +o(1).
\end{align*}
We conclude that at least one of $m^k, m^\ell=0$, and thus no two connected clusters can accumulate at that scale. 
\end{proof}

We remark that we have shown that {\em in the limit} $\eta\to 0$, only a finite number $K<\infty$ of nontrivial connected clusters remain.  It is possible under the hypotheses of Lemma~\ref{ConcLem2} that for $\eta>0$, $\Om_\eta$ has additional (and perhaps an unbounded number of) components with vanishing mass.  For minimizers this should not be the case, but it requires further arguments involving regularity of minimizers (see \cite{ABCT2}) to make this conclusion.


\medskip

\begin{proof}[\bf{Proof of Theorem~\ref{MinThm}}]   Let $v_\eta^{\ast}$ be minimizers of $E_\eta$ with mass $\int_\TT v_\eta^{\ast}=M$.
From the derivation of the Upper Bound in the proof of Theorem~\ref{twodfirst}, we have $\limsup_{\eta\to 0} E_\eta(v_\eta^{\ast})\le \overline{e_0}(M)$.   So by Lemma~\ref{ConcLem1} we obtain the decomposition of $\Om_\eta$ as in that lemma, and the convergence of the minimum values,
$$  \lim_{\eta\to 0} E_\eta(v_\eta^{\ast})=\overline{e_0}(M).  $$
It remains to show that the centers $x_{\eta}^{k}$, $k=1,\dots,K$, remain isolated from each other and in fact converge to minimizers of $\mathcal{F}_K$, where $K$ is determined in Lemma~\ref{ConcLem2}.

For this purpose we must refine the upper and lower bounds on the energy.  Recall from the proof of Lemma~\ref{ConcLem1}, the definition of the clusters $A_\eta^k=\eta^{-1}(\Om_\eta^k- x_\eta^k)$ in $\RR$, with mass $m_\eta^k$, which converge in measure to $A^k$, minimizers of $e_0(m^k)$.  As Lemma~\ref{components} ensures that the components $\Om_\eta^k$ of $\Om_\eta$ are separated by disjoint open sets $\widetilde\Sigma_\eta^k$, we may deform $\Om_\eta$ by moving the centers on $\TT$, while avoiding overlapping. Choose $K$ distinct points $y^k\in\TT$, $k=1,\dots,K$, and define 
$$  \hat\Om_{\eta}^{k}:= \Om_{\eta}^{k}+ y^k -x_{\eta}^{k}= \eta A_{\eta}^{k}+y^k,
$$
and 
$$  w_\eta:=\eta^{-2}\sum_{k=1}^K \chi_{\hat\Om_{\eta}^k} 
   = \sum_{k=1}^K v_\eta^k(\cdot -x_\eta^k + y^k).
$$
Note that if we were to choose $y^k=x_\eta^k$, then $w_\eta=v_\eta$.  Moving the centers of the components leaves the perimeter and self-interaction terms unchanged, but the interaction terms between disjoint components is affected.  Indeed,
\begin{equation}\label{wenergy}
E_\eta(w_\eta)=  \sum_{k=1}^K \EE(A_{\eta}^{k})
  + \sum_{i,j=1}^2 {\Gamma_{ij}\over 2 |\log\eta|}\left[\sum_{k=1}^K J_{i,j}(A_{\eta}^k) 
    +\sum_{k,\ell=1\atop k\neq\ell}^K \int_{A_{i,\eta}^k} \int_{A_{j,\eta}^\ell}
       G_{\TT}(y^k-y^\ell +\eta(x-y))\, dx\, dy
    \right]
\end{equation}

In particular, let $y^1,\dots,y^K$ be minimizers of the point interaction energy,
$$  \mu_K:= 
\min_{\xi^1,\dots,\xi^k\in\TT}\mathcal{F}_K(\xi^1,\dots,\xi^K; \{m^1,\dots,m^k\})
=\mathcal{F}_K(y^1,\dots,y^K; \{m^1,\dots,m^k\}).$$
Because of the logarithmic repulsion at short range, the points $y^1,\dots,y^K$ are distinct and thus well-separated,  $G_\TT$ is smooth in each integral appearing in the last term of  \eqref{wenergy}.  We may thus pass to the limit in this term to obtain: 
\begin{equation}\label{Glimit}
 \int_{A_{i,\eta}^k} \int_{A_{j,\eta}^\ell}
       G_{\TT}(y^k-y^\ell +\eta(x-y))\, dx\, dy - o(1)     \xrightarrow{\eta_{} \rightarrow 0} G_\TT( y^k-y^\ell)\, m_i^k\, m_j^\ell ,  
\end{equation}
to obtain an upper bound,
\begin{align}\nonumber
E_\eta(v_\eta^{\ast})&\le E_\eta(w_\eta) \\
 &= \sum_{k=1}^K \EE(A_{\eta}^{k})
  + \sum_{i,j=1}^2 {\Gamma_{ij}\over 2 |\log\eta|}\sum_{k=1}^K J_{i,j}(A_{\eta}^k) 
     + {\mu_K\over |\log\eta|}
      + o(|\log\eta|^{-1}).
\label{bestupper}
\end{align}

We next claim that the centers $x^1_\eta,\dots,x^K_\eta$ are well-separated: $\exists \delta>0$ for which $r_{\eta}^{k,\ell}\ge\delta>0$, $k\neq\ell$.  By taking a further subsequence (still denoted by $\eta$) if necessary we may then conclude that each sequence $\{x_{\eta}^{k}\}_{}$ converges to a distinct $x^k\in\TT$.
Indeed, assume the contrary, so  $\exists \ k_0\neq\ell_0$ for which $r_{\eta}^{k_0,\ell_0}\to 0$ as $\eta \to 0$.  We then apply the lower bound \eqref{CL18} and \eqref{wenergy} (with $w_\eta=v_\eta$) to derive a lower bound of the form:
\begin{align*}
E_\eta (v_\eta^{\ast})&\ge \sum_{k=1}^K \EE(A_{\eta}^{k}) 
+ \sum_{i,j=1}^2 {\Gamma_{ij}\over 2 |\log\eta|}\sum_{k=1}^K J_{i,j}(A_{\eta}^k)
+ {|\log r_{\eta}^{k_0,\ell_0}|\over |\log\eta|} \sum_{i,j=1}^2 {\Gamma_{ij}\over 4\pi}  |A_{\eta}^{k_0}|\, |A_{\eta}^{\ell_0}| - o(|\log\eta |^{-1}).
\end{align*}
Matching the above lower bound with the upper bound \eqref{bestupper}, we obtain:
$$  |\log r_{\eta}^{k_0,\ell_0}| \sum_{i,j=1}^2 {\Gamma_{ij}\over 4\pi}  |A_{\eta}^{k_0}|\, |A_{\eta}^{\ell_0}| \le C,  $$
with constant $C$ independent of $\eta$.  As 
$|A_{\eta}^{k_0}|\ ,  |A_{\eta}^{\ell_0}|\to m^{k_0},m^{\ell_0}$, which are not both zero, this is impossible.  Thus, each $r_{\eta}^{k_0,\ell_0}$ is bounded away from zero and the claim is verified.

Finally, we show that the limiting locations $x^k=\lim_{\eta\to 0} x_\eta^k$ must minimize $\mathcal{F}_K$.  Using $E_\eta(v_\eta^{\ast})\le E_\eta(w_\eta)$, and writing the energy expansion \eqref{wenergy} for both $v_\eta^{\ast}$ and $w_\eta$ (choosing points $y^k$ which minimize $\mathcal{F}_K$,) all terms cancel exactly except for the interactions, and we are left with:
$$  \sum_{i,j=1}^2  \frac{\Gamma_{ij}}{2}  \sum_{k,\ell=1\atop k\neq\ell}^K \int_{A_{i,\eta}^k} \int_{A_{j,\eta}^\ell}
       G_{\TT}(x_\eta^k-x_\eta^\ell +\eta(x-y))\, dx\, dy
   \le  \sum_{i,j=1}^2  \frac{\Gamma_{ij}}{2}  \sum_{k,\ell=1\atop k\neq\ell}^K \int_{A_{i,\eta}^k} \int_{A_{j,\eta}^\ell}
       G_{\TT}(y^k-y^\ell +\eta(x-y))\, dx\, dy.
$$
As the collections $\{y^1,\dots,y^K\}$ and $\{x^1,\dots,x^K\}$ are well separated on $\TT$, we may pass to the limit as in \eqref{Glimit}:  $|A_{i,\eta}^k\triangle A_i^k|\to 0$, $m^k_{i,\eta}\to m^k_i$, and $x_\eta^k\to x^k$.  By Lebesgue dominated convergence we pass to the limit in each integral to obtain
$$   \mathcal{F}_K(x^1,\dots,x^K; \{m^1,\dots,m^K\})\le \mu_K, $$
and hence $\{x^1,\dots,x^K\}$ are minimizers of $\mathcal{F}_K$.  This completes the proof of Theorem~\ref{MinThm}.
\end{proof}

\begin{remark}
\rm  We note that $J_{i,j}(A^k)$ is related to the constant term in the second $\Gamma$-limit $F_0$, via
$$  \lim_{\eta\to 0} \sum_{i,j =1}^2   \sum_{k=1}^{K} \frac{\Gamma_{ij} }{2} J_{i,j}(A^k)  =
\sum_{i,j =1}^2   \sum_{k=1}^{K}  \frac{\Gamma_{ij} }{2}     \left [   f (m^k_i,m^k_j) +  m_i^k m_j^k R_{\mathbb{T}^2}( 0 )  \right ] ,
$$
with $m^k_i=|A^k_i|$.
\end{remark}

\appendix
\renewcommand{\theequation}{A.\arabic{equation}}
\renewcommand{\thelemma}{A.\arabic{lemma}}
\setcounter{equation}{0}

\vskip 1cm

\noindent{\Large \bf Appendix A}

\vspace{ 0.5cm}

A double bubble is bounded by three circular arcs of radii $r_1, r_2$ and $r_0$, where $r_0$ is the radius of the common boundary of the two lobes of a double bubble. 
Denote by $\theta_1, \theta_2$, and $\theta_0$ the angles associated with the three arcs.

{\em The asymmetric double bubble.}
Given an asymmetric double bubble, we assume without loss of generality that
\begin{eqnarray}
m_1 < m_2. \nonumber 
\end{eqnarray}
We recall that the equations relating masses, angles, and radii are 
\begin{eqnarray}
\label{e1A}
r_1^2 (\theta_1 - \cos \theta_1 \sin \theta_1) + r_0^2 (\theta_0 - \cos \theta_0 \sin \theta_0)  &=& m_1 \\
\label{e2A}
r_2^2 (\theta_2 - \cos \theta_2 \sin \theta_2) - r_0^2 (\theta_0 - \cos \theta_0 \sin \theta_0)  &=& m_2 \\
\label{e3A}
r_1 \sin \theta_1 &=& r_0 \sin \theta_0 \\
\label{e4A}
r_2 \sin \theta_2 &=& r_0 \sin \theta_0 \\
\label{e5A}
(r_1)^{-1} - (r_2)^{-1} &=& (r_0)^{-1}\\
\label{e6A}
\cos \theta_1 + \cos \theta_2 + \cos \theta_0 &=& 0
\end{eqnarray}
From \eqref{e3A}-\eqref{e5A}, we get
\begin{eqnarray} \label{e7A}
\sin \theta_1 - \sin \theta_2 - \sin \theta_0 = 0.
 \end{eqnarray}
Combine \eqref{e6A} with \eqref{e7A}, we get
\begin{eqnarray}
\cos(\theta_1+\theta_0) = - \frac{1}{2},  \text{ and } \cos(\theta_2-\theta_0) = - \frac{1}{2}. \nonumber
\end{eqnarray}
That is,
\begin{eqnarray} \label{e8A}
\theta_1 = \frac{2\pi}{3} - \theta_0,    \text{ and } \;  \theta_2 = \frac{2\pi}{3} + \theta_0.
\end{eqnarray}
Based on \eqref{e8A}, since $\theta_2 < \pi$, we get 
\begin{eqnarray}
0 < \theta_0 < \frac{\pi}{3}. \nonumber 
\end{eqnarray}
Based on \eqref{e1A}, \eqref{e2A} and \eqref{e3A}, we have
\begin{eqnarray}
\frac{   \frac{2\pi/3 - \theta_0}{ \sin^2 (2\pi/3 - \theta_0)}  - \frac{\cos( 2\pi/3 - \theta_0)}{\sin( 2\pi/3 - \theta_0)} + \frac{\theta_0}{\sin^2 \theta_0} - \frac{\cos \theta_0}{\sin \theta_0}  }{    \frac{2\pi/3 + \theta_0}{ \sin^2 (2\pi/3 + \theta_0)}  - \frac{\cos( 2\pi/3 + \theta_0)}{\sin( 2\pi/3 + \theta_0)} - \frac{\theta_0}{\sin^2 \theta_0} + \frac{\cos \theta_0}{\sin \theta_0}  } = \frac{m_1}{m_2}.
\end{eqnarray}
So $\theta_0$ depends on $m_1/m_2$ implicitly. That is, $\theta_0$ is a function of $m_1/m_2$.
Thus, $\theta_1$ and $\theta_2$ are also functions of $m_1/m_2$ due to \eqref{e8A}.

Based on \eqref{e1A}, \eqref{e3A}, and \eqref{e4A}, we get
\begin{eqnarray}
r_0^2 &=&  \frac{m_1}{\sin^2 \theta_0 \left[  \frac{\theta_1}{\sin^2 \theta_1 }  -   \frac{\cos \theta_1}{\sin \theta_1 } +  \frac{\theta_0}{\sin^2 \theta_0 }  -   \frac{\cos \theta_0}{\sin \theta_0 }  \right]  }, \nonumber  \\
r_1^2 &=&  \frac{ m_1}{ \sin^2 \theta_1 \left[  \frac{\theta_1}{\sin^2 \theta_1 }  -   \frac{\cos \theta_1}{\sin \theta_1 } +  \frac{\theta_0}{\sin^2 \theta_0 }  -   \frac{\cos \theta_0}{\sin \theta_0 }  \right]  }, \nonumber  \\
r_2^2 &=&  \frac{m_1 }{ \sin^2 \theta_2  \left[  \frac{\theta_1}{\sin^2 \theta_1 }  -   \frac{\cos \theta_1}{\sin \theta_1 } +  \frac{\theta_0}{\sin^2 \theta_0 }  -   \frac{\cos \theta_0}{\sin \theta_0 }  \right]  }. \nonumber
\end{eqnarray}

Thus the total perimeter of a double bubble is
\begin{eqnarray}
p(m_1, m_2) = 2 \sum_{i=0}^2 \theta_i r_i = \sqrt{m_1}  g \left ( \frac{m_1}{m_2} \right ),   \text{ when } m_1 < m_2, \nonumber 
\end{eqnarray}
where $g$ only depends on the ratio $m_1/m_2$.

{\em The symmetric double bubble.} For an symmetric double bubble, we assume that $ m_1 = m_2 $. 
The middle arc of the double bubble becomes a straight line. $\theta_1 =  \theta_2 = \frac{2\pi}{3}$, and $\theta_0 = 0$. $r_1 = r_2$, $r_0 = \infty$. And we have
\begin{eqnarray}
r_1^2 (\frac{2\pi}{3} - \cos \frac{2\pi}{3} \sin \frac{2\pi}{3} ) = m_1. \nonumber 
\end{eqnarray}
Therefore, 
\begin{eqnarray}
p(m_1, m_2) =  2 \sqrt{2} \sqrt{\frac{4}{3} \pi + \frac{\sqrt{3}}{2}} \sqrt{m_1}. \nonumber
\end{eqnarray}

\end{document}